\documentclass[preprint,3p, times]{elsarticle}

\usepackage{macros}

\journal{Comput.  Methods  Appl.  Mech.  Engrg. }

\begin{document}
	
\begin{frontmatter}
		
\title{A stabilized total pressure-formulation of the Biot's poroelasticity equations in frequency domain: numerical analysis and applications}
    
\author[mymainaddress]{Cristian C\'arcamo\corref{mycorrespondingauthor}}
\cortext[mycorrespondingauthor]{Corresponding author}
\ead{carcamo@wias-berlin.de}

\author[mymainaddress]{Alfonso Caiazzo}
\ead{alfonso.caiazzo@wias-berlin.de}

\author[mysecondaryaddress]{Felipe Galarce}
\ead{felipe.galarce@pucv.cl}

\author[JMaddress]{Joaquín Mura}
\ead{joaquin.mura@usm.cl}

\address[mymainaddress]{Weierstra{\ss}-Institut f\"ur Angewandte Analysis  und  Stochastik,  Leibniz-Institut  im  Forschungsverbund  Berlin  e.V. (WIAS),  Berlin,   Germany}
\address[mysecondaryaddress]{School of Civil Engineering. Pontificia Universidad Cat\'olica de Valpara\'iso, Valpara\'iso, Chile}

\address[JMaddress]{Department of Mechanical Engineering, Universidad Técnica Federico Santa María, Santiago, Chile}

%%%%%%%%%%%%%%%%%%%%%%%%%%%%%%%%%%%%%%%%%%
% ABSTRACT
\begin{abstract}
This work focuses on the numerical solution of the dynamics of a poroelastic material in the frequency domain. We provide a detailed stability analysis based on the application of the Fredholm alternative in the continuous case, considering a total pressure formulation of the Biot's equations.
In the discrete setting, we propose a stabilized equal order finite element method complemented by an additional pressure stabilization to enhance the robustness of the numerical scheme with respect to the fluid permeability.
Utilizing the Fredholm alternative, we extend the well-posedness results to the discrete setting, obtaining theoretical optimal convergence for the case of linear finite elements.
We present different numerical experiments  to validate the proposed method. First, we consider model problems with known analytic solutions in two and three dimensions. As next, we show that the method is robust for a wide range of permeabilities, including the case of discontinuous coefficients. 
Lastly, we show the application for the simulation of brain elastography on a realistic brain geometry obtained from medical imaging.
\end{abstract}
%%%%%%%%%%%%%%%%%%%%%%%%%%%%%%%%%%%%%%%%%

\end{frontmatter}

%\tableofcontents

\section{Introduction}
This paper focuses on the simulation of poroelastic materials following Biot's equations \cite{BIOT41}, in which the 
interplay of bulk deformation, fluid flow,
and fluid pressure is modeled coupling linear elasticity with
a flow through a deformable porous media.
This model has been widely applied in diverse fields ranging from hydrology and geomechanics (see, e.g., \cite{zhang2021discrete}) to biomechanics \cite{sowinski2020poroelasticity}), and fluid transport in soft tissue such as perfusion \cite{perfusionBrainMRI2007,capillarPerm2012}

Our work is motivated by the application of poroelastic modeling for the solution of inverse problems in tissue imaging, in particular in Magnetic Resonance Elastography (MRE) (see, e.g., \cite{hirschbraunsack-17,sack-review-2023}), an acquisition technique which is sensitive to tissue motion.
\rrevision{In MRE, the tissue undergoes a mechanical harmonic excitation at given frequencies (typically in the range 1--100Hz), applied on the external surface of the body.
The internal tissue displacement field is reconstructed in vivo via phase-contrast MRI. Combining the reconstructed displacement field with a physical tissue model allows us to gain insights into relevant biomechanical parameters.
Established application of MRE based on linear elastic or viscoelastic models include 
the estimation of tissue stiffness to support the noninvasive diagnosis and staging of pathologies such as cancer and fibrosis, as well as the characterization of cancer tissue properties (see, e.g., \cite{sack-review-2023}).

Recent research in MRE focuses on the usage of 
poroelastic tissue models 
elastography in characterizing tissue biphasic properties \cite{lilaj_etal_2021a} and interstitial pressure \cite{hirsch-etal-2014-liver}. 
Preliminary applications and computational methods for addressing inverse problems in poroelastography
can be found, e.g., in \cite{perriñez2010magnetic} (solution of the inverse problem for elastic parameters
within a poroelastic tissue model), \cite{Tan-etal-2017}
(domain decomposition method to address the lack of
information on tissue pressure, and \cite{GCS-2023}
(assimilation of displacement data in a poroelastic brain model).

This work addresses the numerical analysis of the mathematical model underlying poroelastography. From this perspective}, there are several works focusing on suitable numerical methods for poroelastic materials, including standard Galerkin method \cite{MULO92}, adaptive algorithms \cite{KHSI20}, mixed variational formulations through the introduction of a Lagrange multiplier and related \cite{AHRANO19, LEE16, YI14}, Discontinuous Galerkin \cite{PHWH08}, adaptive strategies  (also for multiple-network poroelasticity equations) \cite{ERT23, LIZI20, RDGK2017}, highlighting also new methods facilitating the use of general meshes such as Hybrid High Order (HHO) method \cite{BOBODI26}  or Virtual Element Method (VEM) \cite{BURU21}.
Additionally, in \cite{bertrand-2022}, an overview of the Theory of Porous Media restricted to small deformations and its discretization is provided.

We propose and analyze a numerical scheme in the frequency domain based on equal-order finite elements. This choice allows us to maintain low computational costs also in three dimensions. The scheme uses a displacement-pressure-total pressure formulation, equipped with a residual-based stabilization term, inspired by the work of \cite{RORR16} in the static setting, which ensures stability between the space of displacements and the space of total pressures.

The main contribution of this work concerns the detailed numerical analysis, in the continuous and the discrete settings.  Using the Fredholm alternative, we extend the results of \cite{RORR16} to the frequency domain, showing that the total pressure formulation is stable under the assumption of stability of the underlying elastic problem. In particular, we show that the operator defining the differential problem can be written as a compact perturbation of a bijective one (see, e.g., \cite{evans10,rero04,kre77}).
% {\color{blue}Regarding the well-known Fredholm alternative method, we turn to give a brief explanation. In short, the fact that this method implies that the problem consisting of a bijective operator together with a compact perturbation is well-posed when the solution of the associated homogeneous case is the trivial one (see \cite{evans10,rero04,kre77}).}
 % such as shale rocks or mudrocks, which can act as barriers for fluid flow in underground storage or extraction sites \cite{rezaeyan2022mudrocks,zhang2021discrete}. Another example of low permeability porous media is the nucleus pulposus of the intervertebral disc, which is a gel-like substance that contains a high amount of proteoglycans, which bind water molecules and generate a swelling pressure that supports the spine \cite{rohlmann2006analysis,galbusera2011comparison}. 
 
One of the most difficult scenarios to deal with is the case of 
low permeability regions. In those situations, so-called poroelastic locking might result in nonphysical fast pressure oscillations, which can be cured using particular finite element spaces \cite{phillips2009overcoming,RORR16}. In the context of inverse problems, where the parameters are unknown a priori, it is therefore of utmost importance to consider a numerical method that can robustly handle the appearance of low permeability regions throughout the domain.  To this purpose, we propose an additional pressure stabilization, which introduces an additional control on the pressure gradient. The stabilization term, inspired by a Brezzi-Pitkäranta stabilization \cite{brepit84} acts as an artificial local permeability when the physical permeability becomes too low.

We benchmark the proposed method in several numerical tests, 
validating the expected convergence orders, as well as the robustness of the formulation for low permeabilities.

The rest of the paper is organized as follows.  Section \ref{sec:model} introduces the model problem.
The analysis in the continuous case is presented in 
Section \ref{sec:analysis}, while Section
\ref{sec:discrete-analysis} discusses the proposed numerical method and the extension of the 
well-posedness analysis in the case 
of the considered stabilized finite element formulation.
The numerical results are presented in Section 
\ref{sec:results}, while Section 
\ref{sec:conclusions} draws the conclusions.

% References:
% https://www.frontiersin.org/articles/10.3389/fphy.2020.617582/full \cite{sowinski2020poroelasticity}: Take the parameters for the tissue case

%%%%%%%%%%%%%%%%%%%%%%%%%%%%%%%%%%%%%%%%%%
\section{Model Problem}\label{sec:model}
\subsection{Linear poroelasticity in the harmonic regime}
Poroelasticity describes the coupled motion of solid matrix deformation and fluid flow in a porous medium. The equations governing the dynamics consist of a balance of linear momentum for the solid phase, a mass conservation equation for the fluid phase, and a constitutive relation that relates the stress and strain in the solid phase to the fluid pressure. 

Following Biot's theory (see, e.g., \cite{BIOT41,SHOW00,MELLA2023116386}), we consider the motion of a poroelastic medium in a sufficiently regular computational domain $\Omega \subset \mathbb R^d$. The medium is described by
a displacement field $\uu: \Omega \to \mathbb R^d$ and a pressure field $p: \Omega \to \mathbb R$ both serving as solutions to:
\begin{equation} \label{eq:biot}
\left\{
\begin{aligned}
%\rho \utt -\ddiv \sigma +\alpha \nabla p & = \zero \quad \text{in}\,\,\OO\times [0,T]\\
\rho \utt -\ddiv \sigma + \nabla p & = \zero \quad \text{in}\,\,\OO\times [0,T]\\
S_{\epsilon} p_t +\alpha\ddiv \ut -\frac{\kappa}{\mu_f} \Delta p &= 0 \quad \text{in}\,\,\OO\times [0,T]\\
\end{aligned}
\right.
\end{equation}
In \eqref{eq:biot}, the symbol $\sigma$ represents the Cauchy solid stress, defined as
\begin{equation} \label{eq:cauchy-stress}
\sigma := 2\mu_e\eps({\uu})+\lambda(\nabla \cdot \uu)\II,
\end{equation}
where $\eps(\uu)$ is the infinitesimal strain tensor, $\II$ represents the identity tensor, and the Lamé coefficients are given by
\[
\mu_e:= \frac{E}{2(1+\nu)},\,\, \lambda:= \frac{E\nu}{(1+\nu)(1-2\nu)}, 
\]
as functions of the Young modulus $E$ and of the Poisson ratio $\nu$. 
	
The parameter $\kappa$ in \eqref{eq:biot} represents the permeability of the porous medium, while $\mu_f,\rho >0$ denote the fluid viscosity and density, respectively. Additionally, $\alpha>0$ is the Biot-Willis parameter, $B$ is the so-called Skempton's parameter, and the mass storage parameter $\Se$ is defined as
$$
\Se=3\alpha(1-\alpha B)(1-2\nu)(BE)^{-1}.
$$

Following the approach of \cite{RORR16}, we then introduce the \textit{total pressure}
\begin{equation} \label{eq:biot-fft0}
\phi =   p-\lambda\tr(\eps(\uu)).
\end{equation}

\rrevision{This work is motivated by applications in elastic tissue imaging, such as MRE \cite{GCS-2023,hirschbraunsack-17,sack-review-2023}), where the material undergoes a harmonic mechanical  excitation at a moderate - given - frequency (1--100Hz), imposed on the external tissue surface. In particular, MRE requires
the solution of an inverse problem for the estimation of
relevant parameters based on the tissue response, in terms of displacement field, to the harmonic excitation.
Targeting the solution of the corresponding forward problem,} we hence focus on system \eqref{eq:biot} in the harmonic regime for a given frequency $\omega$:

\begin{equation} \label{eq:biot-fft}
		\left\{
			\begin{aligned}
				-\omega^2 \rho \uu -\ddiv (2\mu_e\eps(\uu)-\phi\II)& = \zero \\
				i\left(\Se+\frac{\alpha}{\lambda}  \right)\,\omega p -i\omega\frac{\alpha}{\lambda} \phi-\frac{\kappa}{\mu_f}\Delta p &= 0, \\
				\phi - p + \lambda \tr(\epsilon(\uu)) &= 0\,.
			\end{aligned}
		\right.
	\end{equation}
With a slight abuse of notation, we will denote the (complex-valued) $\omega$-Fourier modes of velocity, pressure, and total pressure as $\uu$, $p$, and $\phi$, respectively, while $i$ represents the imaginary unit.

Moreover, we introduce the (dimensionless) parameter
\begin{equation}\label{eq:theta}
\theta : = \frac{\Se \lambda}{\alpha} + 1
= \frac{3 \nu \alpha (1- \alpha B)}{\alpha B(1+\nu)(1-2\nu)^2}+ 1.
\end{equation}

The system \eqref{eq:biot-fft} shall be complemented by appropiate boundary conditions on the displacement and pressure fields.
Throughout the rest of this work, we assume that the boundary of the domain is decomposed as
$$
\partial \OO=\Gamma_{\uu}  \cup \Gamma_{p}.
$$   
Denoting $\nn$ as the outward 
normal vector to the boundary, we consider boundary conditions of the form
\begin{equation}\label{eq:biot_bc_u}
\left\{
\begin{aligned}
u&= \zero \quad \text{on}\,\,\GU \\
\sigma \nn &= \bg^{\uu} \quad \text{on}\,\,\GP\\
\end{aligned}
\right.
\end{equation}
for the displacement, and
\begin{equation}\label{eq:biot_bc_p}
		\left\{
		\begin{aligned}
			p &  =0 \quad \text{on}\,\,\GP\\
			\frac{\kappa}{\mu_f} \partial_{\nn} p&= g^p \quad \text{on}\,\,\GU \\
		\end{aligned}
		\right.
	\end{equation}
for the pressure. \rrevision{In \eqref{eq:biot_bc_u}--\eqref{eq:biot_bc_p},
the terms $\bg^{\uu}$ and $g^p$ denote the harmonic forces and fluid pressures, with frequency $\omega$, imposed on
$\GP$ and $\GU$, respectively.}

\subsection{Weak formulation}
Let us consider the standard Sobolev spaces
$L^2(\OO)$ and $H^1(\OO)$ of complex-valued functions equipped with the inner products
\begin{equation}\label{eq:l2-prod}
(u,v)_{\OO} := (u,v)_{L^2(\OO)} :=   \int_{\OO} u\,\overline{v},
\end{equation}
and 
\begin{equation}\label{eq:h1-norm}
(u,v)_{1,\OO} := (u,v)_{H^1(\OO)} = (u,v)_{\OO}+ \ell^2 \sum^d_{i=1}  \bigg(\frac{\partial u}{\partial x_i},\frac{\partial v}{\partial x_i} \bigg)_{\OO},
	\end{equation}
 respectively, where $\overline{v}$ stands for the complex conjugate of $v$.
In \eqref{eq:h1-norm}, the parameter $\ell$ denotes a typical length of the domain $\OO$, and it has been introduced for the purpose of maintaining consistency in physical units.

Let us also denote with $\Vert \,\cdot\,\Vert_0$ and $\Vert \,\cdot\,\Vert_1$
the standard norms induced by the above inner products, and introduce the seminorm
\[ 
\vert v\vert^2_1 := \sum^d_{i=1} \bigg\Vert\frac{\partial u}{\partial x_i} \bigg\Vert^2_{\OO},
\]
such that $\Vert v \Vert_1 = \Vert v \Vert_0 + \ell^2 \vert v \vert_1$, for any $v \in H^1(\OO)$.

For any subset in $\Gamma \subset \partial \Omega$, we also denote by $L^2(\Gamma)$ the space of integrable functions on $\Gamma$ and by $\langle \cdot ,\cdot \rangle_{\Gamma}$ the corresponding inner product.

In the above setting, let us introduce the functional spaces
\begin{equation}\label{eq:spaces}
\begin{aligned}
\HH & :=  \{\vv:\Omega \to \mathbb C^d,\,\vv\in H^1(\OO)^d:\vv=\zero \,\,\text{on}\,\,\GU\}.\\
P & := \{q:\Omega \to \mathbb C,\, q\in H^1(\OO): q=0 \,\,\text{on}\,\, \GP  \}\\
S & := L^2(\OO),
\end{aligned}
\end{equation}
as well as the product space $\UU:=\HH\times P\times S$, 
equipped with the norm
\begin{equation}\label{norm:cont}
\Vert (\vv,q,\xi) \Vert^2_{\UU}:= 			
2 \mu_e \nnormz{\eps(\vv)}^2 + \frac{\kappa}{\mu_f \omega \,\alpha}\nnorm{q}{1}^2
+ \lambda^{-1}\nnormz{\xi}^2\,.
\end{equation}

As next, we introduce the bilinear forms
\begin{eqnarray}
a_1: \HH \times \HH \to \mathbb C, \; a_1(\uu,\vv) &: =& -\omega^2 \rho \uprod{\uu}{\vv} + 2\mu_e\uprod{\eps(\uu)}{\eps(\vv)}  \label{eq:a_1_form} \\
a_2: P \times P \to \mathbb C, \; a_2(p,q)	&:= &  i \theta \lambda^{-1} \uprod{ p}{q}
+ \frac{\kappa}{\mu_f \omega \alpha} \uprod{\nabla p}{\nabla q} \label{eq:a_2_form}\\
\tilde b: \HH \times S \to \mathbb C, \; \tilde b(\uu,\xi) & :=&  -\uprod{\ddiv \uu}{\xi} \label{eq:b_form}\\
\tilde b^*: S \times \HH \to \mathbb C, \; \tilde b^*(\phi,\vv) & :=&  - \uprod{\phi}{\ddiv \vv} \label{eq:b_t_form}\\
c: P \times S \to \mathbb C, \;  c(p,\xi) & :=& -\lambda^{-1}\uprod{p}{\xi} \label{eq:c_form}\\
c^*: S \times P \to \mathbb C, \;  c^*(\phi,q) & :=& -\lambda^{-1}\uprod{\phi}{q} \label{eq:c_t_form}\\
d: S \times S \to \mathbb C, \; d(\phi,\xi) &:=& \lambda^{-1}\uprod{\phi}{\xi}.
 \label{eq:d_form}
\end{eqnarray}
 
Multiplying the equations of system \eqref{eq:biot-fft} by
$\vv \in \HH$, $q \in P$, and $\xi \in S$, respectively, integrating by parts, and imposing the boundary conditions
\eqref{eq:biot_bc_u}-\eqref{eq:biot_bc_p}, we consider problem: Find $\uve = (\uu,p,\phi) \in \UU$ such that
	\begin{equation}\label{weakF}
		\langle \mathcal{A}(\uve),\vve \rangle = \langle \mathcal{F},\vve \rangle
	\end{equation}
for all $\vve = (\vv,q,\xi) \in \UU$, where 
$\mathcal{A}: \UU \to \UU^*$ is defined by
\begin{equation}\label{eq:Auv-weak}
\begin{aligned}
\langle \mathcal{A}(\uve),\vve \rangle & : = 
a_1(\uu,\vv) + \tilde b^*(\phi,\vv)
+a_2(p,q) + ic^*(\phi,q) + d(\phi,\xi) + c(p,\xi)
-\tilde b(\uu,\xi)\,,
\end{aligned}
\end{equation}
and $\mathcal F \in \UU^*$ is defined by
\begin{equation}\label{eq:Fv-weak}
\begin{aligned}
\langle \mathcal{F},\vve \rangle :=&\, \langle {\bg}^{\uu},\vv\rangle_{\GP} +\frac{1}{\omega}\langle {g}^p,q\rangle_{\GU}\,.
\end{aligned}
\end{equation}

We also introduce the operators 
$\mathcal B:\UU \to \UU^*$ and $\mathcal C:\UU \to \UU^*$ defined by
\begin{equation}\label{eq:op-B}
\langle \mathcal{B}(\uve),\vve \rangle  : = 
a_1(\uu,\vv) + \tilde b^*(\phi,\vv)
+a_2(p,q) + d(\phi,\xi) -\tilde b(\uu,\xi)\,,
\end{equation}
and
\begin{equation}\label{eq:op-C}
\langle \mathcal{C}(\uve),\vve \rangle : = 
ic^*(\phi,q) + c(p,\xi)\,,
\end{equation}
respectively. These operators allow us to rewrite
\begin{equation}\label{eq:A-decomp}
\mathcal A = \mathcal B + \mathcal C\,.
\end{equation}

The decomposition \eqref{eq:A-decomp} will be utilized to establish the well-posedness of the weak formulation 
\eqref{eq:Auv-weak}-\eqref{eq:Fv-weak} by employing Fredholm's alternative (see, e.g., \cite{evans10}). The alternative states that
an operator is bijective if it can be written 
as a sum of a  bijective operator and a compact operator.

\section{Analysis of the continuous problem}\label{sec:analysis}

\subsection{Preliminaries}\label{ssec:preliminaries}
To begin with, let us recall few essential theoretical results \revision{and concepts} which will be required for
the upcoming analysis.
\begin{lemma}[Poincar\'e inequality] \label{poincare}
Let There exists a positive constant $C_P$ , depending on $\OO$, such that
\begin{equation}\label{eq:poincare}
\Vert q \Vert_{1} \leq C_P \ell^2 \,\vert q \vert_1,
\end{equation}
for all $q \in H^1(\Omega)$. 
\end{lemma}
% \begin{proof}
%     See, e.g., \cite{BRSC07}.
% \end{proof}
\revision{We observe that, in particular,
\eqref{eq:poincare} implies that
the norm \eqref{norm:cont} controls also the $L^2$-norm of the displacement.}

\begin{lemma}[Trace inequality]\label{trace}
Assuming that $\Omega$ has a Lipschitz boundary and $p \in \mathbb R$, with $1 \leq p \leq \infty$. There exists a constant $C>0$ such that
\begin{equation}\label{eq:trace}  
		\Vert \vv \Vert_{0,\Gamma}=\sqrt{\langle \vv,\vv\rangle_{\Gamma}} \leq C \Vert \vv\Vert^{\frac{1}{2}}_0 \vert \vv\vert^{\frac{1}{2}}_1,
\end{equation}
		for all $\vv \in H^1(\Omega)^d$\,.
	\end{lemma}
% \begin{proof}
%     See, e.g.,  \cite{BRSC07}.
% \end{proof}
%	
\begin{lemma}[Korn inequality] \label{korn}
There exist a positive constant $C_K$, dependent on $\OO$, such that
\begin{equation}\label{eq:korn}
\Vert \vv\Vert_1 \leq C_K \,\Vert \eps(\vv)\Vert_0,
\end{equation}
for all $\vv \in H^1(\Omega)^d$\,.
\end{lemma}

We refer the reader to \cite{BRSC07,ciar13} for detailed proofs.
% \begin{proof}
%     See, e.g.,  \cite[Theorem 6.15-4]{ciar13}.
% \end{proof}

As next, we introduced a weaker definition of coercivity.
\begin{definition}[$T$-coercivity]\label{def:T-coer}
Let $(V, \langle \cdot, \cdot \rangle_V)$ and $(W, \langle \cdot, \cdot \rangle_W)$ be two Hilbert spaces. A linear operator $L: V \to W^*$ is called \textit{$T-$coercive} if there exists $T\in \mathcal{L}(V,W)$ bijective and a constant $\tilde{\alpha} >0$, such that
		\[
		\vert \langle L(v),T(v) \rangle \vert \geq \tilde{\alpha} \Vert v \Vert^2_V\,,
		\]
  holds for all $v \in V$.
\end{definition}

As demonstrated in \cite{ciar12}, the property of \textit{$T-$coercive} is sufficient
to establish the well-posedness of the corresponding bilinear form.
\begin{theorem}\label{theo:T-coer}
Let $L: V \to W^*$ be a linear operator, and let $\langle L(u),v \rangle $ be the induced bilinear form over the product space $V\times W$. Then, the following statements are equivalent:
\begin{itemize}
\item[i)] The problem  $\langle L(u),v\rangle = \langle f,v\rangle$ is well-posed, for any $f \in W$
\item[ii)] $L$ is \textit{$T-$coercive}.
\end{itemize}
\end{theorem}
For the proof, we refer the reader to \cite{ciar12}.
	
Finally, the following result will be used to show the well-posedness of the 
variational problem, exploiting the structure of the operators in the
product space $\UU = \HH \times P \times S$.
\begin{theorem} \label{theo:bijec}
Let $(V, \langle \cdot, \cdot \rangle_V)$ and $(Z, \langle \cdot, \cdot \rangle_Z)$ be two Hilbert spaces and let us consider a linear operator 
$\mathsf{T}: V\times Z \to V^* \times Z^*$ on the product space that can be written in the form
\begin{equation}\label{eq:theorem-T-mat}
\mathsf{T}(v,z) = 
\begin{pmatrix}
	A & B^*  \\ B & C
 \end{pmatrix} \begin{pmatrix} v \\z \end{pmatrix} = (Av+B^* z, Bv+Cz)\,
\end{equation}
for bounded linear operators 
$A : V \to V^*$, $B : V \to Z^*$, and $C : Z \to Z^*$. 
Assume that:
\begin{itemize}
	\item[i)] $A$ is elliptic, i.e., there exists $\alpha > 0$ such that $\langle Av, v \rangle_V \geq \alpha \lVert v \rVert^2_{V}$ for all $v \in V$,
	\item[ii)] $B$ is surjective, i.e., there exists $\beta > 0$ such that $\lVert B^* z \rVert_V \geq \beta \lVert z \rVert_Z$ for all $z \in Z$,
	\item[iii)] $C$ is positive semidefinite, i.e., $\langle Cz,z \rangle \geq 0\qquad \forall z\in Z$.
\end{itemize}
Then, $\mathsf{T}$ is bijective. 
\begin{proof}
 See \cite[Lemma 3.4 ]{GAOYSA11} and \cite[Lemma 2.1]{GAHEME03}.
 \end{proof}

\end{theorem}

\subsection{Well-posedness}\label{ssec:cont-stab}
\rrevision{Our analysis focuses on the numerical properties of the proelasticity problem \eqref{eq:biot-fft}, and it is built upon the key assumption that the underlying elasticity equation is well-posed at the continuous level in the space $\HH$.}

Let us define the scalar products on $\HH$ as follows:
\begin{align*}
		& (\vv,\ww)_{0,\rho} := \rho(\vv,\ww)_{\OO} \\
		&(\vv,\ww)_{1,\mu_e} := 2\mu_e(\epsilon(\vv),\epsilon(\ww))_{\OO} \\
		& (\vv,\ww)_{1,\mu_e,\rho} :=  (\vv,\ww)_{1,\mu_e} + (\vv,\ww)_{0,\rho}
\end{align*}	
Here $\vv, \ww \in \HH$, and denote the associated norms as $\Vert\vv\Vert_{0,\rho}$, $\Vert \vv\Vert_{1,\mu_e}$, and $\Vert\vv\Vert_{1,\mu_e,\rho}$.

Given that $\OO$ is bounded and assuming that the boundary $\partial \OO$ is sufficiently regular, one can conclude that $\HH$ is compactly embedded into $L^2(\OO)^d$ (see, e.g., \cite{evans10}).  
Therefore, there exists a Hilbert basis of $L^2(\OO)^d$ composed of eigenfunctions
of the elasticity operator, i.e., there 
exists a family $ (\vv_n,\lambda_n)_n \in \HH\times\RR^+$ such that $\vv_n \neq \mathbf{0}$
and 
\begin{equation}\label{eq:eigenvalues}
\begin{aligned}
 & (\vv_n,\ww)_{1,\mu_e} = \lambda_n\, (\vv_n,\ww)_{0,\rho} \quad \forall \, \ww\,\in\,\HH, \\
 & \lim_{n \to \infty} \lambda_n = +\infty,\\
 & \Vert \vv_n \Vert_{1,\mu_e,\rho} = 1\,.
 \end{aligned}
 \end{equation}
 Hence, for any $\vv\in\HH$, it holds
	\[ \vv = \sum_{n\geq 0} \alpha_n \vv_n, \;
 \alpha_n  := (\vv,\vv_n)_{1,\mu_e,\rho},\; \]
 and
 $\Vert \vv\Vert_{1,\mu_e,\rho}^2 =  \displaystyle \sum_{n\geq 0} \alpha^2_n $.

\rrevision{Following \cite{CIA10}, we then assume 
that 
\begin{equation}\label{eq:w-assumpt-1}
\omega^2 \notin (\lambda_n)_{n\geq 0}
\end{equation}
which guarantees the well-posedness of the 
underlying elasticity problem.}
\rrevision{
As in \cite{CIA10}, let us also define 
\begin{equation}\label{eq:m_max}
\overline{m} := \max\{ n \in \mathbb N \mid \omega^2 >\lambda_n\}.
\end{equation}
\begin{remark}
    The above assumption \eqref{eq:w-assumpt-1} has been used also in \cite{CIA10} in the context of Helmholtz equation. We observe that proving \eqref{eq:w-assumpt-1} in general requires the solution of a second-order eigenvalue problem, which, although interesting, lies out of the scopes of this work.
\end{remark}}

Firstly, we show the continuity of $\mathcal{A}$ and $\mathcal{F}$ in the chosen norm.
\begin{lemma}[Continuity] \label{cont}
There exist two constants $\eta_1,\eta_2>0$, depending on the physical and geometrical problem parameters such that 
\begin{equation}\label{eq:Acont}
\Vert \mathcal{A}(\uve)\Vert_{\UU} \leq \eta_1\, \Vert \uve  \Vert_{\UU}
\end{equation}
and
\begin{equation}\label{eq:Fcont}
\Vert \mathcal{F}  \Vert_{\UU}\leq \eta_2 \,\bigg(\Vert \bg^{\uu}_{re} \Vert_{0,\GP}+ \Vert g^p_{re} \Vert_{0,\GU}\bigg)
\end{equation}
\end{lemma}
\begin{proof}
The results follow from the Cauchy-Schwarz inequality and from the inequality \eqref{eq:trace}. One obtains:
\begin{equation}\label{eq:A-continuity}
			\begin{aligned}
				\left| \langle \mathcal{A}(\uve),\vve \rangle \right| \leq & 
				\omega^2 \rho \nnormz{\uu}\nnormz{\vv}
				+ 2\mu_e \nnormz{\eps(\uu)}\nnormz{\eps(\vv)}
				+ \frac{\theta}{\lambda} \nnormz{p}\nnormz{q}+ \frac{\kappa}{\mu_f \omega\,\alpha}\nnormz{\nabla p}\nnormz{\nabla q} \\
				&      + \nnormz{\phi}\nnormz{\nabla \cdot \vv} + 
				 \nnormz{\phi}\nnormz{q} + \nnormz{\nabla \cdot \uu}\nnormz{\xi} + \lambda^{-1}\nnormz{\phi}\nnormz{\xi} + \lambda^{-1} \nnormz{p}\nnormz{\xi}\\
				\leq & \,\,\bigg(\frac{\omega^2\rho(C_P C_K)^2}{2\mu_e}+1\bigg)\,\,2\mu_e\Vert \eps(\uu)\Vert_{0}\Vert \eps(\vv)\Vert_0+\bigg(\frac{\theta}{\lambda}\bigg(\frac{\kappa}{\mu_f\omega\,\alpha}\bigg)^{-1}+C_P\bigg)\frac{\kappa}{\mu_f\omega\,\alpha} \Vert p \Vert_1 \Vert q \Vert_1 \\
				&+ \frac{\lambda^{1/2}C_K\sqrt{d}}{\sqrt{2\mu_e}} \lambda^{-1/2}\Vert\phi\Vert_0\sqrt{2\mu_e}\Vert \eps(\vv)\Vert_0 +\alpha \lambda^{1/2}\bigg(\frac{\kappa}{\mu_f\omega\,\alpha}\bigg)^{-1/2} \lambda^{-1/2}\Vert \phi \Vert_0 \bigg(\frac{\kappa}{\mu_f\omega\,\alpha}\bigg)^{1/2}\Vert q \Vert_1 \\
				& + \frac{\lambda^{3/2}C_K \sqrt{d}}{\sqrt{2\mu_e}} \sqrt{2\mu_e}\nnormz{\eps(\uu)} \lambda^{-1/2}\nnormz{\xi}+ \lambda^{-1}\nnormz{\phi}\nnormz{\xi}\\
				&+\alpha\lambda^{-1/2}\bigg(\frac{\kappa}{\mu_f\omega\,\alpha}\bigg)^{-1/2} \bigg(\frac{\kappa}{\mu_f\omega\,\alpha}\bigg)^{1/2}\Vert p\Vert_1\lambda^{-1/2}\nnormz{\xi}\\
				\leq & \, \eta_1 \,\Vert \uve \Vert_{\UU} \,\Vert \vve\Vert_{\UU}
			\end{aligned}
		\end{equation}
		where
		\begin{equation}\label{eq:eta_1}
			\eta_1 = 3 \max\left\lbrace \frac{\omega^2\rho(C_P C_K)^2}{2\mu_e}+1,   \frac{\theta}{\lambda}\bigg(\frac{\kappa}{\mu_f\omega\,\alpha}\bigg)^{-1}+C_P, \frac{\lambda^{1/2}C_K\sqrt{d}}{\sqrt{2\mu_e}},  \lambda^{1/2}\bigg(\frac{\kappa}{\mu_f\omega\,\alpha}\bigg)^{-1/2},  \frac{\lambda^{3/2}C_K \sqrt{d}}{\sqrt{2\mu_e}}, \lambda^{-1/2}\bigg(\frac{\kappa}{\mu_f\omega\,\alpha}\bigg)^{-1/2} \right\rbrace  \,.
		\end{equation}

For the right hand side, it holds
\begin{align*}
			\langle \mathcal{F},\vve \rangle  \leq & \Vert \bg^{\uu} \Vert_{0,\GP} \Vert \vv \Vert_{0,\GP}+\Vert g^{p} \Vert_{0,\GU} \Vert q \Vert_{0,\GU}\\
			\leq & C^1_{tr}\Vert \bg^{\uu} \Vert_{0,\GP} \vert \vv \vert_{1}+\frac{C^2_{tr}}{\omega}\Vert g^{p} \Vert_{0,\GU} \Vert q \Vert_{1},
\end{align*}
where we have used the inequalities \eqref{eq:poincare} and \eqref{eq:trace}.  
The estimate \eqref{eq:Fcont} follows then from the Cauchy-Schwarz inequality, i.e., 
		\begin{align*}
			\langle \mathcal{F},\vve \rangle  \leq & \eta_2\,\bigg(\Vert \bg^{\uu}\Vert^2_{0,\GP} +\Vert g^{p} \Vert^2_{0,\GU}\bigg)^{1/2} \bigg(\vert \vv \vert^2_{1}+ \Vert q \Vert^2_{1}\bigg)^{1/2}\leq \eta_2\,\bigg(\Vert \bg^{\uu} \Vert_{0,\GP} +\Vert g^{p} \Vert_{0,\GU}\bigg)\Vert \vve \Vert_{\UU}\,,
		\end{align*}	
		where 
		\begin{equation}
			\eta_2 = C_{tr} \max \left\lbrace (2\mu_e)^{-1/2},\bigg(\frac{\kappa}{\mu_f\omega\alpha}\bigg)^{-1/2}\frac{1}{\omega^2}\right\rbrace . 
		\end{equation}
	\end{proof}

\revision{\begin{remark}[Role of the physical parameters]\label{rem:stab_kappa}
    The continuity constants depend on the physical parameters. In particular, 
    from \eqref{eq:eta_1}, it follows that
    \[
    \eta_1 = O(\omega^2) + O\left(\frac{\omega}{\kappa}\right).
    \]
    The stability can therefore deteriorate in case of very large frequencies or very small permeabilities. Notice that the present work is motivated by application
    in \rrevision{elastic imaging (elastogaphy), where the the frequency of the mechanical excitation is given a priori and whose range is typically moderate (1--100 Hz) \cite{sack-review-2023}}. However, 
    the presence of small permeabilities can introduce stability issues in the discrete setting. This point will be discussed in Section \ref{sec:discrete-analysis}.
\end{remark}
}

Let $(\vv_n)$ be the eigenvectors introduced in 
\eqref{eq:eigenvalues}. Let us now consider the index
$\overline{m}$ introduced in \ref{eq:m_max} and the subspace
\begin{equation*}
\HH^-:= \Span_{0\leq n\leq \overline{m}} (\vv_n)\,.
\end{equation*}
Let $\PHm$ be the orthogonal projection on $\HH^-$
and let $\mathbb{T} := \II_{\HH} - 2\PHm$, where
$\II_{\HH}$ is the identity on $\HH$.

%%%%%%%%%%%%%%%%%%%%%%%%%
% T-Coercivity of A_1
%%%%%%%%%%%%%%%%%%%%%%%%%
\begin{lemma}\label{lemma:T-coer}
Let $a_1(\cdot,\cdot)$ be the bilinear form introduced in \eqref{eq:a_1_form}.
Under the hypotheses of \eqref{eq:w-assumpt-1} and \eqref{eq:m_max} it holds
\begin{itemize}
\item[(i)] $a_1(\cdot,\cdot)$ is $\mathbb{T}$-coercive, and 
\item[(ii)] the bilinear form 
$(\uu,\vv) \mapsto a_1(\uu,\PHm(\vv )) $ is positive definite, i.e.,
\begin{equation}\label{eq:a1_P_pos}
a_1(\vv,\PHm(\vv)) > 0\,,\forall \vv \in \HH\,.
\end{equation}
\end{itemize} 
\end{lemma}
% Proof
\begin{proof}
The proof follows the approach presented in \cite{ciar12}.  First, it is noteworthy that, based on the definition of $\mathbb{T}$, one can derive
\begin{equation*}
\mathbb{T} \vv_n =
\begin{cases}
-\vv_n, & 0\leq n\leq \overline{m}\\
+\vv_n, & n>\overline{m}.
\end{cases}
\end{equation*}
Thus, $\mathbb{T}^2 = \II$, implying that $\mathbb{T}$ is bijective. 
For the $\mathbb T$-coercivity, it shall be proven that
there exists a constant $\alpha_{min}$ depending on $\omega$ and $(\lambda_n)_{n\geq 0}$ such that
\[ 
a_1(\vv,\mathbb{T}(\vv)) \geq  \alpha_{min} \,\Vert\vv\Vert^2_{1,\mu_e,\rho}
\]
for all $\vv \in \HH$.

To this end, we follow \cite[Prop. 1]{ciar12}, which allows to obtain
\begin{equation}\label{eq:elasticity_eigen}
\begin{aligned}
a_1(\vv,\mathbb{T}(\vv))   & =  \sum_{0\leq n\leq \overline{m}} \alpha_n a_1(\vv,(\mathbb{T}(\vv_n)) + \sum_{ n > \overline{m}}\alpha_n a_1(\vv,(\mathbb{T}(\vv_n))  = 
\sum_{0\leq n\leq \overline{m}} \alpha_n a_1(\vv,-\vv_n) + \sum_{ n > \overline{m}}\alpha_n a_1(\vv,\vv_n)  \\
& = \sum_{0\leq n\leq \overline{m}} \alpha_n  \left[ \omega^2(\vv,\vv_n)_{0,\rho}-(\vv,\vv_n)_{1,\mu_e} \right] + \sum_{ n> \overline{m}} \alpha_n \left[ (\vv,\vv_n)_{1,\mu_e}- \omega^2(\vv,\vv_n)_{0,\rho} \right] \\
& = \sum_{0\leq n\leq \overline{m}}  \left(\frac{\omega^2-\lambda_n}{1+\lambda_n} \right)\,\alpha^2_n  + \sum_{ n > \overline{m}}  \left(\frac{\lambda_n-\omega^2}{1+\lambda_n} \right)\,\alpha^2_n \geq  \alpha_{min} \,\Vert\vv\Vert^2_{1,\mu_e,\rho} 
\end{aligned}
\end{equation}
with $\alpha_{min} = \min_{n\geq 0}  \left\vert\frac{\omega^2-\lambda_n}{1+\lambda_n} \right\vert$.

The inequality \eqref{eq:a1_P_pos} can be demonstrated
using analogous steps. One obtains
\begin{equation}
\begin{aligned}
a_1(\vv,\PHm(\vv))   & =  
\sum_{0\leq n\leq m} \alpha_n a_1(\vv,(\PHm(\vv_n))   	= \sum_{0\leq n\leq m} \alpha_n a_1(\vv,\vv_n) 	= \\
& = \sum_{0\leq n\leq m} \alpha_n  \left[ \omega^2(\vv,\vv_n)_{0,\rho}-(\vv,\vv_n)_{1,\mu_e} \right] =  \sum_{0\leq n\leq m}  \left(\frac{\omega^2-\lambda_n}{1+\lambda_n} \right)\,\alpha^2_n  \geq   \tilde{\alpha}_{min} \,\sum_{0\leq n\leq m}  \alpha^2_n >0,
\end{aligned}
\end{equation}
where $\tilde{\alpha}_{min} = \min_{n\geq 0}  \frac{\omega^2-\lambda_n}{1+\lambda_n}$,
and $\tilde{ \alpha}_{min} \neq 0$ due to \eqref{eq:w-assumpt-1}, \eqref{eq:m_max}, and the properties \eqref{eq:eigenvalues}.

\end{proof}

% A is elliptic
\begin{lemma}\label{lemma:A-elliptic}
Let $a_1(\cdot,\cdot)$ and $a_2(\cdot,\cdot)$ be the bilinear forms
introduced in equations \eqref{eq:a_1_form} and \eqref{eq:a_2_form}, respectively.
The bilinear form $a(\cdot,\cdot): (\HH \times P) \times (\HH \times P) \to \mathbb C$, defined by
\[ 
a((\uu,p);(\vv,q)): = a_1(\uu,\vv) + a_2 (p,q),
\]
for $(\uu,p),(\vv,q) \in \HH \times P$, is elliptic.
\end{lemma}
\begin{proof}
Let $(\vv,q) \in \HH\times P$.
From $\mathbb T + 2\PHm = \II_{\HH}$ it follows that
\begin{equation}
\re a_1(\vv,\vv) = \re a_1(\vv,\mathbb{T}(\vv)) + 
2\re a_1(\vv,\PHm(\vv) ,
\end{equation}
for all $\vv \in \HH$.
Using Lemma \ref{lemma:T-coer} and inequality \eqref{eq:poincare} one obtains
\begin{equation}
\re a_1(\vv,\vv) + 	\re a_2(q,q)= 
\alpha_{min} 2 \mu_e \Vert \eps(\vv)\Vert^2_{0}	+ \frac{\kappa}{\omega \,\alpha\, \mu_f} \Vert \nabla q \Vert^2_0   \geq \tilde{\alpha} \bigg(2 \mu_e \Vert \eps(\vv)\Vert^2_{0}	+ \frac{\kappa}{\mu_f\, \omega\, \alpha} \Vert  q \Vert^2_1 \bigg), 
\end{equation}
	% 	\begin{align*}
	% 		\re \langle \mathbb{A}_1(\vv),\vv \rangle + 	\re \langle \mathbb{A}_2(q),q \rangle = 	\re  \mathbf{A}_1(\vv,\vv) + 	\re  \mathbf{A}_2(q,q)   = 	\re  \mathbf{A}_1(\vv,\mathbb{T}(\vv))+2\re \mathbf{A}_1(\vv,\mathbb{P}(\vv))+ \re  \mathbf{A}_2(q,q). 
	% 	\end{align*}
	% Then, by using Lemma \eqref{lemma:T-coer}, for all $(\vv,q) \in \HH\times P$	 we get
	% \begin{align*}
	% \re \langle \mathbb{A}_1(\vv),\vv \rangle + 	\re \langle \mathbb{A}_2(q),q \rangle = \alpha_{min} 2 \mu_e \Vert \eps(\vv)\Vert^2_{0}	+ \frac{\kappa}{\omega \,\alpha\, \mu_f} \Vert \nabla q \Vert^2_0   \geq \alpha_1 \bigg(2 \mu_e \Vert \eps(\vv)\Vert^2_{0}	+ \frac{\kappa}{\mu_f\, \omega\, \alpha} \Vert  q \Vert^2_1 \bigg), 
	% \end{align*}
with 
\begin{equation}\label{eq:alpha_u}
\tilde{\alpha} = \min \left\{ \alpha_{min},C_P^{-1}\right\}\,.
\end{equation}
\end{proof}

% B satisfies inf-usp
\begin{lemma}\label{lemma:inf-sup-B}
The bilinear form $\tilde b$ defined in \eqref{eq:b_form} satisfies a continuous inf-sup condition, i.e., there exists $\beta_1 >0$ such that
\begin{equation}\label{ISC}
\sup_{\vv \in \HH \atop \vv\neq \vec{\bf O}} \frac{ \tilde{b}(\vv,\xi) }{\vert \vv \vert_1} \geq \beta_1 \Vert \xi\Vert_0,
\end{equation}
for all $\xi \in S$.
\end{lemma}
\begin{proof}	
See, e.g., \cite{GIRA79}.
\end{proof}

The previous results allow to prove the first  main result.

% Bijectivity of \mathcal B
\begin{lemma}\label{lemma:WPB}
The operator $\mathcal{B}$ defined in equation \eqref{eq:op-B} is bijective.
\end{lemma}
\begin{proof}
The proof relies on decomposing
the operator $\mathcal B$ as in 
\eqref{eq:theorem-T-mat}.
We observe that, for all $\xi \in S$ it holds
\[ 
\re  d(\xi,\xi) = 
\lambda^{-1}\Vert \xi \Vert^2_0 \geq 0. 
\]
	
Combining this result with Lemma \ref{lemma:A-elliptic} and Lemma
\ref{lemma:inf-sup-B} allows us to infer bijectivity 
of $\mathcal B$ using Theorem \ref{theo:bijec}.
\end{proof}

% Compactness of \mathcal C
\begin{lemma}\label{lemma:compact}
The operator $\mathcal{C}$, defined in equation \eqref{eq:op-C}, is compact.
\end{lemma}
\begin{proof}
The compactness of $\mathcal C$ follows from 
the fact that $\mathbb{C} = \lambda^{-1} I \circ i_c$, where $I: L^2(\OO)\rightarrow L^2(\OO)$ and $i_c$ represents the identity operator along with the compact embedding from $H^1(\OO)$ into $L^2(\OO)$ (for details, see \cite[Lemma 2.2]{RORR16}).
\end{proof}
		
% Iniectivity of \mathcal A
\begin{lemma}\label{lemma:InjecBC}
Under the hypothesis of  \eqref{eq:w-assumpt-1} and \eqref{eq:m_max}, the operator $\mathcal{A}$ is injective. 
\end{lemma}
\begin{proof}
Let $\vve \in \UU$ such that $\mathcal{A}(\vve)=0$, i.e., \
\[
\re \langle  \mathcal{A}(\vve),\vve \rangle =\im \langle \mathcal{A}(\vve),\vve \rangle=0\,.
\] 
From
\begin{align*}
\langle \mathcal{A}(\vve),\vve \rangle  &  = -\omega^2 \rho \Vert \vv\Vert^2_0 + 2\mu_e\Vert \eps(\vv)\Vert^2_0  +i \theta \lambda^{-1}
			\Vert q \Vert^2_0 + \frac{\kappa}{\mu_f \omega \alpha} \Vert \nabla q \Vert^2_0 \\
			& \quad  
			- i  \lambda^{-1}\uprod{\xi}{q}  - \lambda^{-1}\uprod{q}{\xi} +\lambda^{-1}\Vert \xi \Vert^2_0.
		\end{align*}
one obtains ($\lambda \in \mathbb R)$
\begin{equation*}
\begin{aligned}
\im \langle \mathcal{A}(\vve),\vve \rangle    = 0 &  \Leftrightarrow   \theta \lambda^{-1}\Vert q \Vert^2_0	- \lambda^{-1} \re \uprod{\xi}{q}  - \lambda^{-1}\im \uprod{q}{\xi} = 0\,,
		\end{aligned}
  \end{equation*}
and hence, using $\uprod{q}{\xi} = \uprod{\xi}{q}$,
\begin{equation}\label{eq:imA}
\lambda^{-1}\im \uprod{\xi}{q} = \lambda^{-1}\theta \Vert q \Vert^2_0 - \lambda^{-1}\re \uprod{\xi}{q}
\end{equation}
Analogously, from $\re \langle \mathcal{A}(\vve),\vve \rangle   =  0 $, one obtains
\begin{equation}\label{eq:reA}
\begin{aligned}
0 & \,\,\,\,=  -\omega^2 \rho \Vert \vv\Vert^2_0 + 2\mu_e\Vert \eps(\vv)\Vert^2_0  + \frac{\kappa}{\mu_f \omega \alpha} \Vert \nabla q \Vert^2_0	+  \lambda^{-1} \im\uprod{\xi}{q}  -  \lambda^{-1} \re\uprod{\xi}{q} +\lambda^{-1}\Vert \xi \Vert^2_0   \\
&       \underbrace{=}_{\eqref{eq:imA}} -\omega^2 \rho \Vert \vv\Vert^2_0 + 2\mu_e\Vert \eps(\vv)\Vert^2_0  + \frac{\kappa}{\mu_f \omega \alpha} \Vert \nabla q \Vert^2_0	+  \lambda^{-1}\theta \Vert q \Vert^2_0 - 2\lambda^{-1}\re \uprod{\xi}{q}  +\lambda^{-1}\Vert \xi \Vert^2_0  \\
& \underbrace{=}_{\eqref{eq:theta}} -\omega^2 \rho \Vert \vv\Vert^2_0 + 2\mu_e\Vert \eps(\vv)\Vert^2_0  + \frac{\kappa}{\mu_f \omega \alpha} \Vert \nabla q \Vert^2_0	+  \frac{\Se}{\alpha}\Vert q \Vert^2_0  +  \lambda^{-1} \Vert q \Vert^2_0 - 2\lambda^{-1}\re \uprod{\xi}{q}  +\lambda^{-1}\Vert \xi \Vert^2_0  \\
&\,\,\,\, = -\omega^2 \rho \Vert \vv\Vert^2_0 + 2\mu_e\Vert \eps(\vv)\Vert^2_0  + \frac{\kappa}{\mu_f \omega \alpha} \Vert \nabla q \Vert^2_0	+  \frac{\Se}{\alpha}\Vert q \Vert^2_0  +  \lambda^{-1} \left(
\Vert q \Vert^2_0 - 2\re \uprod{\xi}{q}  +\Vert \xi \Vert^2_0 \right)\,.
\end{aligned}
  \end{equation}

Hence, from
\begin{equation}\label{eq:re-xi-q}
\re (\xi,q ) \leq \Vert \xi \Vert_0 \Vert q \Vert_0 \leq \frac{1}{2}\Vert \xi \Vert^2_0 +\frac{1}{2}\Vert q \Vert^2_0 
\end{equation}
and Lemma \ref{lemma:A-elliptic}
one obtains
\begin{equation}
0 \geq 2\mu_e\alpha_{\min}\Vert\eps(\vv) \Vert^2_0  
+ \frac{\kappa}{\mu_f \omega \alpha} \Vert \nabla q \Vert^2_0	
+\frac{\Se}{\alpha} \Vert q \Vert^2_0 
  \end{equation}
which is satisfied only for $(\vv,q)= (\zero,0)$. At the same time,  
$(\vv,q)= (\zero,0)$ yields
\[
0 = \langle \mathcal{A}(\vve),\vve \rangle = \lambda^{-1} \Vert \xi \Vert^2_0
\]
and thus $\xi =0$, concluding the proof.
\end{proof}

Using Lemmas \ref{lemma:WPB}, \ref{lemma:compact}, \ref{lemma:InjecBC} and the Fredholm’s alternative \cite[Theorem 4.2.9]{Sauter2011} allows to state the main stability result.
\begin{theorem}[Well-posedness]
The problem \eqref{weakF} has a unique solution $\uve^* \in \UU$, and there exists a positive constant $C$ such that there holds 
\begin{equation}\label{eq:unique}\Vert \uve^* \Vert_{\UU} \leq C \Vert \mathcal{F} \Vert_{\UU} \leq C \bigg[\Vert g^p \Vert_{0,\GU}+\Vert \bg_{re} \Vert_{0,\GP}  \bigg]. 
\end{equation}
\end{theorem}
	
\begin{remark}
Note that \eqref{eq:unique} is equivalent to the following inf–sup condition:
\begin{equation}\label{eq:inf-sup}
\exists \,\beta_2>0 :\;		
\inf_{\uve\in\UU \atop \uve \neq \vec{\bf 0}}\, \sup_{\vve \in \UU \atop \vve\neq \vec{\bf O}} 
\frac{\left|\langle \mathcal{A}(\uve),\vve\rangle\right| }{\Vert \vve \Vert_{\UU} \Vert \uve \Vert_{\UU}} \geq \beta_2\,.
\end{equation}
 \end{remark}

\section{Analysis of the discrete problem}\label{sec:discrete-analysis}
This section is dedicated to the well-posedness and stability analysis of the discrete problem arising using a stabilized finite element formulation of \eqref{weakF}. 	

\subsection{Stabilized finite element formulation}\label{ssec:fe-formulation}

Let $\{\Trih\}_{h>0}$ denote a shape-regular triangulation of $\overline{\OO}$. 
For an element $T \in \mathcal T_h$ we denote with $h_T$ the diameter, introducing
\[
h:=\max \{h_T:T\in \Trih\}
\]
as the characteristic mesh size. Let us also assume that there exist
a constant $h_0>0$ such that $h \leq h_0$, for all triangulations.

Denoting $\mathcal{P}_j(T)$ ($j \in \mathbb{N}$) as the space of polynomials of total degree less than or equal to $j$ over an element $T \in \mathcal{T}_h$, we define the following continuous\revision{, equal-order}, finite element spaces:
\begin{equation}\label{eq:spaces-h}
\begin{aligned}
    \HH_h := &\left\lbrace \vh \in C(\overline{\Omega})^d : \vh\vert_T \in \mathcal{P}_k(T)^d, \, \forall T \in \Trih \right\rbrace \cap \HH \\
    P_h := &\left\lbrace \qh \in C(\overline{\Omega}) : \qh\vert_T \in \mathcal{P}_\revision{k}(T), \, \forall T \in \Trih \right\rbrace \cap P \\
    S_h := &\left\lbrace \xih \in C(\overline{\Omega}) : \xih\vert_T \in \mathcal{P}_\revision{k}(T), \, \forall T \in \Trih \right\rbrace 
\end{aligned}
\end{equation}
and let $\UU_h \equiv \HH_h \times P_h \times S_h$.
%The analysis presented in the following part can be applied to general finite element triples. However, we focus on the case of equal-order elements, i.e., choosing $k=l=m$.

%
It is well known that, \revision{in the case of equal-order elements}, the discrete spaces do not satisfy an inf-sup condition. For this reason, the discrete formulation will be equipped with additional stabilizations. On the other hand, the choice of equal order elements is motivated by the reduced computational cost, particularly evident in realistic three-dimensional examples.

\revision{We} consider the residual of the momentum equation:
\begin{equation}\label{eq:res_u}
\mathbf{R}(\vh, \xih) := \omega^2 \rho \vh + 2\mu_e\ddiv \eps(\vh) - \nabla \xih.
\end{equation}
Additionally, we introduce an additional term inspired by the Brezzi-Pitkäranta stabilization (see \cite{brepit84}) and define the operator $\mathcal{S}_h: \UU_h \rightarrow \UU_h^*$ as follows:
\begin{equation}
\langle \mathcal{S}_h(\uvh), \vvh \rangle \defi \delta_1 \sum_{T \in \Trih} h^2_T (\mathbf{R}(\uh, \phih), \mathbf{R}(\vh, \xih))_{T} + \delta_2\sum_{T \in \Trih} \frac{h^2_T}{\mu_f\, \alpha \omega}(\nabla \ph, \nabla \qh)_{T},
\label{eq:S}
\end{equation}
where $\delta_1 > 0$ and $\delta_2 \geq 0$ are two stabilization parameters.

\rrevision{The first term is designed to address the lack of inf-sup stability in the finite element spaces (see, e.g., \cite{RORR16,phillips2009overcoming,vuong-tesis}). The second term can be seen as an artificial permeability which is motivated by the lack of control
on the pressure error in the norm \eqref{norm:cont} for very low values of permeability, and might therefore become relevant only for $\kappa \ll 1$ (see also Remark \ref{rem:stab_kappa}). Concrete practical examples will be provided in 
Section \ref{sec:results}}.

%The first term is designed to handle the lack of inf-sup stability of the finite element spaces, whilst the second addresses instabilities arising in the low permeability range (see e.g., \cite{phillips2009overcoming}, \cite{vuong-tesis}).
%
%\begin{remark}
% For clarity, the Brezzi-Pitkäranta term is only necessary for small values of $\kappa$, i.e., when $\kappa \ll 1$. Otherwise, it suffices to consider $\delta_2 = 0$.
%\end{remark}	
	
The proposed finite element formulation reads: 
\begin{problem}\label{pb:disc}
Find $\uvh = (\uh, \ph, \phih) \in \UU_h$ such that
	\begin{equation}\label{stab}
		\langle \mathcal{A}_h(\uvh), \vvh \rangle = \langle \mathcal{F}, \vveh \rangle, \quad \forall \vveh \in \UU_h,
	\end{equation}
	where 
	\begin{equation}\label{eq:Ah-decomp}
		\mathcal{A}_h := \mathcal{B} + \mathcal{S}_h + \mathcal{C}\,.
	\end{equation} 
\end{problem}

\subsection{Well-posedness of the discrete problem} 
The well-posedness of problem \eqref{stab} will be addressed based on the decomposition \eqref{eq:Ah-decomp}, following an argument analogous to the one used in Section \ref{sec:analysis}.

First, let us define the following mesh-dependent norm over $\UU_h$:
\begin{equation}\label{norm:disc}
	\Vert \vveh \Vert^2_{\UU_h} := \Vert \vveh \Vert^2_{\UU} + \delta_1 \sum_{T \in \Trih} h^2_T \Vert \mathbf{R}(\vh, \xih) \Vert^2_{0,T} + \delta_2 \sum_{T \in \Trih} \frac{h^2_T}{\mu_f \, \alpha \omega}  \Vert \nabla q \Vert^2_{0,T}.
\end{equation}

In what follows, we will also use the following inverse inequalities: there exist two constants $C_I$ and $\tilde{C}_I$ such that
\begin{equation}\label{eq:inverse-1} 
h^2_T \Vert \ddiv \eps(\vh) \Vert^2_{0,T} \leq C_I^2 \Vert \eps(\vh) \Vert^2_{0,T}\,,
\end{equation}
and
\begin{equation}\label{eq:inverse-2}
h^2_T \Vert \nabla \vh \Vert^2_{0,T} \leq \tilde{C}_I^2 \Vert \vh \Vert^2_{0,T}\,,
\end{equation}
for any element $T$ in the triangulation and for all $\vh \in \HH_h$.

We begin stating a result analogous to Theorem \ref{theo:T-coer}, valid for the
discrete setting.
\begin{theorem}\label{theo:Th-coer}
Let $(V_h)_h$ and $(W_h)_h$  be two families of finite dimensional Hilbert spaces such that $\text{dim} V_h = \text{dim} W_h$, $\forall h$, and let $(L_h)_h$ a family of operators $L_h:V_h \to W_h$, uniformly bounded in $h$.
Then, the followings statements are equivalent:
\begin{itemize}
			\item[(i)] The problem  $L_h u_h = f$ is well-posed and $(L^{-1}_h)_h$ is uniformly bounded;
			%\item[(ii)] $L_h$ satisfies a discrete inf-sup condition on $V_h\times W_h$.
			\item[(ii)] $(L_h)_h$ is $T$-coercive.
		\end{itemize}
		\end{theorem}
	\begin{proof}
		See \cite[Th. 2]{ciar12}.
	\end{proof}

The following lemma concerns the orthogonality of the Galerkin finite element method, which is only achieved asymptotically ($\sim O(h^2)$) or when $\delta_2 = 0$.

\begin{lemma}[$h^2$-Galerkin Orthogonality]\label{eq:ort}
Let $\uve$ and $\uveh$ be the solutions of \eqref{weakF} and \eqref{stab}, respectively. Assume that $\uu\,\in \,\HH\cap H^2(\Omega)^d$, $p\,\in \,p\cap H^2(\Omega)$, and $\phi \,\in\,S\cap H^1(\Omega)$. Then, it holds 
\begin{equation}\label{eq:h-orthogonality} 
\langle \mathcal{A}_h(\uve-\uveh),\vveh\rangle =\delta_2\sum_{T \in \Trih} \frac{h^2_T}{ \mu_f\,\alpha \omega}  (\nabla p,\nabla \qh)_{T},
\end{equation}
for all $\vveh \in \UU_h$. 
\end{lemma}

\begin{proof}
It holds
\begin{align*}
\langle \mathcal{A}_h(\uve-\uveh),\vveh\rangle =\langle \mathcal{A}(\uve),\vveh\rangle+\langle \mathcal{S}_h(\uve),\vveh\rangle -\langle \mathcal{A}_h(\uveh),\vveh\rangle = \mathcal{S}_h(\uve)\,.
    %\delta_2\sumkth\frac{\mu_f \,\alpha}{\kappa} h^2_T (\nabla p,\nabla \qh )_{T}.
			\end{align*}
Using the assumption on the regularity of the
solution of \eqref{weakF}, i.e., $\uu\,\in \,\HH\cap H^2(\OO)^d$ and $\phi \,\in\,S\cap H^1(\OO)$, one
can conclude that $\omega^2\rho \uu+2\mu_e\ddiv \eps(\uu)-\nabla \phi=0$, and hence
   \[ \mathcal{S}_h(\uve)=\delta_2\sumkth\frac{h^2_T}{\mu_f\,\alpha \omega}  (\nabla p ,\nabla \qh )_{T}.\] 
% \begin{align*}
% 				\langle \mathcal{A}_h(\uve-\uveh),\vveh\rangle =\langle \mathcal{A}(\uve),\vveh\rangle+\langle \mathcal{S}_h(\uve),\vveh\rangle -\langle \mathcal{A}_h(\uveh),\vveh\rangle = \delta_2\sumkth\frac{\mu_f \,\alpha}{\kappa} h^2_T (\nabla p,\nabla \qh )_{T}.
% 			\end{align*} 
		\end{proof}
		
The next lemma shows the continuity of the
stabilized finite element operator $\mathcal A_h$.
	
\begin{lemma}[Continuity of $\mathcal A_h$]\label{contS}
There exist two positive constants $\eta_3$ and $\eta_4$ such that, for any $\uveh\in\UU_h$, we have
\[ \Vert \mathcal{S}_h(\uveh)\Vert \leq \eta_3 \,\Vert \uveh\Vert_{\UU} \quad \text{and}\quad \Vert \mathcal{A}_h(\uveh)\Vert \leq \eta_4 \,\Vert \uveh\Vert_{\UU}. \]
\end{lemma}

\begin{proof}
		Let $\uveh,\vveh \in \UU_h$. Using the inverse inequalities \eqref{eq:inverse-1} ad \eqref{eq:inverse-2}
  we obtain
		% \begin{equation}\label{eq:inverse} h^2_T\Vert \ddiv \eps(\vh) \Vert^2_{0,T}\leq C^2_I\Vert \eps( \vh)\Vert^2_{0,T} \qquad \text{and} \qquad  h^2_T\Vert \nabla  \vh \Vert^2_{0,T}\leq \tilde{C}^2_I\Vert \vh\Vert^2_{0,T}\end{equation}
		\begin{equation} \label{eq:boundS}
			\begin{split}
				 & \delta_1 \sumkth h^2_T (\omega^2\rho \uh+2\mu_e\ddiv \eps(\uh)-\nabla \phih , \omega^2\rho \vh+2\mu_e\ddiv \eps(\vh)-\nabla \xih)_{T}\\
				\leq & \delta_1 \sumkth h^2_T\bigg(\omega^2 \rho \Vert\uh \Vert_{0,T}+2\mu_e \Vert \ddiv \eps(\uh)\Vert_{0,T}+\Vert \nabla \phih\Vert_{0,T}\bigg)\bigg(\omega^2 \rho \Vert\vh \Vert_{0,T}+2\mu_e \Vert \ddiv \eps(\vh)\Vert_{0,T}+\Vert \nabla \xih\Vert_{0,T}\bigg)\\
				\leq & \delta_1 \sumkth \bigg(\omega^2 \rho \,h_T\,\Vert\uh \Vert_{0,T}+2\mu_e\, C_I\,\Vert \eps(\uh)\Vert_{0,T}+\tilde{C}_I\Vert  \phih\Vert_{0,T}\bigg)\bigg(\omega^2 \rho h_T\Vert\vh \Vert_{0,T}+2\mu_e C_I\Vert  \eps(\vh)\Vert_{0,T}+  \tilde{C}_I\Vert  \xih\Vert_{0,T}\bigg)\\
				\leq & \delta_1 \max \left\lbrace \omega^2\rho\, h_0 (2\mu)^{-1/2},(2\mu_e)^{1/2}\, C_I, \tilde{C}_I \lambda^{1/2}\right\rbrace^2\bigg[\sumkth  \bigg((2\mu_e)^{1/2}\Vert\uh \Vert_{0,T}+(2\mu_e)^{1/2}\Vert \eps(\uh)\Vert_{0,T}+\lambda^{-1/2}\Vert  \phih\Vert_{0,T}\bigg)^2\bigg]^{1/2}\times \\
				&\bigg[\sumkth \bigg((2\mu_e)^{1/2}\Vert\vh \Vert_{0,T}+(2\mu_e)^{1/2}\Vert \eps(\vh)\Vert_{0,T}+\lambda^{-1/2}\Vert  \xih\Vert_{0,T}\bigg)^2\bigg]^{1/2}\\
				\leq & 3\delta_1 \max \left\lbrace \omega^2\rho\, h_0 (2\mu)^{-1/2},(2\mu_e)^{1/2}\, C_I, \tilde{C}_I \lambda^{1/2}\right\rbrace^2\bigg[\sumkth  \bigg(2\mu_e\Vert\uh \Vert^2_{0,T}+2\mu_e\Vert \eps(\uh)\Vert^2_{0,T}+\lambda^{-1}\Vert  \phih\Vert^2_{0,T}\bigg)\bigg]^{1/2}\times \\
				&\bigg[\sumkth \bigg(2\mu_e\Vert\vh \Vert^2_{0,T}+2\mu_e\Vert \eps(\vh)\Vert^2_{0,T}+\lambda^{-1}\Vert  \xih\Vert^2_{0,T}\bigg)\bigg]^{1/2}\\
				= & 3\delta_1 \max \left\lbrace \omega^2\rho\, h_0 (2\mu)^{-1/2},(2\mu_e)^{1/2}\, C_I, \tilde{C}_I \lambda^{1/2}\right\rbrace^2\bigg( 2\mu_e\Vert\uh \Vert^2_{0}+2\mu_e\Vert \eps(\uh)\Vert^2_{0}+\lambda^{-1}\Vert  \phih\Vert^2_{0}\bigg)^{1/2} \times \\
				&\bigg( 2\mu_e\Vert\vh \Vert^2_{0}+2\mu_e\Vert  \eps(\vh)\Vert^2_{0}+ \lambda^{-1} \Vert  \xih\Vert^2_{0}\bigg)^{1/2}\\
				\leq & 3\delta_1 \max \left\lbrace \omega^2\rho\, h_0 (2\mu)^{-1/2},(2\mu_e)^{1/2}\, C_I, \tilde{C}_I \lambda^{1/2}\right\rbrace^2\bigg( 2\mu_e (C_P C_K)^2\Vert\eps(\uh) \Vert^2_{0}+2\mu_e\Vert \eps(\uh)\Vert^2_{0}+\lambda^{-1}\Vert  \phih\Vert^2_{0}\bigg)^{1/2} \\
				& \bigg( 2\mu_e (C_P C_K)^2\Vert \eps(\vh) \Vert^2_{0}+2\mu_e\Vert  \eps(\vh)\Vert^2_{0}+ \lambda^{-1} \Vert  \xih\Vert^2_{0}\bigg)^{1/2}\\
				\leq & 3\delta_1 \max \left\lbrace \omega^2\rho\, h_0 (2\mu)^{-1/2},(2\mu_e)^{1/2}\, C_I, \tilde{C}_I \lambda^{1/2}, \right\rbrace^2 (1+(C_P C_K)^2)\Vert\uveh \vert_{\UU}\Vert\vveh \Vert_{\UU}.
			\end{split}
		\end{equation}

		On the other hand, 
\begin{equation}\label{eq:boundP}
\begin{split}
\delta_2 \sumkth \frac{h^2_T}{\mu_f\,\alpha \omega} (\nabla \ph,\nabla \qh)_T &
     \leq  \delta_2 \sumkth \frac{h^2_T}{\mu_f\,\alpha \omega} \Vert \nabla\ph\Vert_{0,T} \Vert\nabla \qh\Vert_{0,T} \\	
     &\leq    \delta_2 \left(\frac{h_0^2}{\kappa}\right) \frac{\kappa}{\mu_f\,\alpha \omega}  \bigg(\sumkth  \Vert \nabla\ph\Vert^2_{0,T} \bigg)^{1/2}\bigg(\sumkth \Vert\nabla\qh\Vert^2_{0,T} \bigg)^{1/2} \\
     %
	%			&\leq  \delta_2 \bigg(\frac{\mu_f\,\alpha}{\kappa}\bigg)^2\omega\, h^2_0 \frac{\kappa}{\mu_f\omega\alpha} \Vert \nabla \ph \Vert_0\Vert \nabla \qh \Vert_0 \\
%
&\leq \delta_2 \left(\frac{h_0^2}{\kappa}\right) \Vert \uveh \Vert_{\UU}\Vert \vveh \Vert_{\UU}.
			\end{split}
		\end{equation}
  So, taking $\eta_3 = 3\delta_1 \max \left\lbrace \omega^2\rho\, h_0 (2\mu)^{-1/2},(2\mu_e)^{1/2}\, C_I, \tilde{C}_I \lambda^{1/2} \right\rbrace^2 (1+(C_P C_K)^2)+\delta_2 \left(\frac{h_0^2}{\kappa}\right)$, we arrive to 
  \begin{equation}\label{ine:S}
  	\Vert \mathcal{S}_h(\uveh) \Vert \leq \eta_3 \Vert \uveh \Vert_{\UU}
  \end{equation}
Inequality \eqref{ine:S} proves the continuity of $\mathcal S_h$. Combined with the continuity of $\mathcal{A}$ shown in Lemma \eqref{cont}, we can conclude that $\mathcal{A}_h$ is also continuous, i.e., 
\begin{equation}\label{contAS}
\langle \mathcal{A}_h(\uveh),\vveh \rangle \leq \eta_4 \Vert \uveh \Vert_{\UU} \Vert \vveh \Vert_{\UU},
\end{equation}
with $\eta_4= \eta_1+\eta_3$.  

\end{proof}

From continuity, it follows that the discrete norm \eqref{norm:disc} is equivalent to the continuous norm \eqref{norm:cont}. In fact, the inequality $\Vert \vve \Vert^2_{\UU_h} \geq \Vert \vve \Vert^2_{\UU}$ is straightforward. On the other hand, by taking $h\leq h_0$ and appealing to the inequalities \eqref{eq:boundS} and \eqref{eq:boundP}, we obtain
	
\begin{equation}\label{norm:equiv}
\Vert \vveh \Vert^2_{\UU_h} = \Vert \vveh \Vert^2_{\UU} + \delta_1\sum_{T \in \Trih} h^2_T \Vert \mathbf{R}(\vh,\phih)\Vert^2_{0,T} + \delta_2\sum_{T \in \Trih} \frac{h_T^2}{\mu_f\,\alpha \omega} \Vert \nabla q \Vert^2_{0,T} \leq (1+\eta_3) \,\Vert \vveh \Vert^2_{\UU}.
\end{equation}

As in the continuous case, the stability of the finite element method relies on the decomposition of the operator $\mathcal A_h$ into an elliptic $\mathcal B + \mathcal S_h$ and a compact operator $\mathcal C$. The compactness of $\mathcal C$ in the discrete case follows with an argument analogous to Lemma \ref{lemma:compact}. The next lemma shows the coercivity of the operator $\mathcal B + \mathcal S_h$.

\begin{lemma}[Coercivity of $\mathcal B + \mathcal S_h$] \label{lemma:coer_Bs}
Let $\mathcal B_h: = \mathcal B + \mathcal S_h$. 
There exists a positive constant $\alpha_1$, which depends on the frequency $\omega$ and on the domain $\Omega$, such that
\begin{equation}\label{eq:coerAS}
\re \langle \mathcal{B}_h(\vveh),\vveh \rangle  \geq  \alpha_1 \Vert \vveh \Vert^2_{\UU_h}.
\end{equation}
\end{lemma}

\begin{proof}
Due to the fact that $\HH^-$ is of finite dimension, it is possible to construct a space $\HH^-_h$ such that the latter is an approximation of the former. This can be achieved by considering approximations $(\vv^h_n)_{0\leq n\leq m}$ of the basis $(\vv_n)_{0\leq n\leq m}$. Therefore, we can define the space 
\[ \HH^-_h = \text{span}_{0\leq n\leq m} (\vv^h_n). \]
Now, similar to the continuous case, we define $\mathbb{T}_h:=\II_{\HH_h}-2\mathbb{P}^-_h$ of $\mathcal{L}(\HH_h)$, whose properties are studied in \cite{ciar12}.

By applying Theorem \eqref{theo:Th-coer} for the discrete $\mathbb{T}_h$-coercivity of $\mathbb{A}_1$, we obtain
\begin{align*}
\re \langle \mathcal{B}_h(\vveh),\vveh \rangle = & \re \langle \mathbb{A}_1(\vh),\vh \rangle+ \re \langle \mathbb{A}_2(\qh),\qh \rangle + \re \langle \tilde{\mathbb{D}}(\xih),\xih \rangle +  \re \langle \mathcal{S}(\vveh),\vveh \rangle \\
\geq & \tilde{\alpha}  \bigg(2\mu_e\Vert \eps(\vvh)\Vert^2_0 + \frac{\kappa}{\mu_f \omega \alpha}\Vert  \qh\Vert^2_1\bigg)+\lambda^{-1}\Vert \xih\Vert^2_0 +\delta_1 \sum_{T \in \Trih} h^2_T \Vert \mathbf{R}(\vh,\xih)\Vert^2_{0,T} + \delta_2 \sum_{T \in \Trih} \frac{h^2_T} {\mu_f\,\alpha \omega} \Vert \nabla \qh \Vert^2_{0,T} \\
\geq & \alpha_1 \Vert \vveh\Vert^2_{\UU_h},
\end{align*}
\rrevision{defining $\alpha_1 := \min\{\tilde{\alpha},1\}$, where $\tilde{\alpha}= \min\{\min_{n\geq 0}  \frac{\omega^2-\lambda_n}{1+\lambda_n},C_P\}$ is the constant defined in \ref{eq:alpha_u}.}

\end{proof}

\begin{remark}
\rrevision{Note that the discrete $\mathbb{T}_h-$operator defined in the proof of the Lemma \ref{lemma:coer_Bs} is similar to the continuous operator $\mathbb{T}$ but defined on discrete spaces.
Although the analysis is omitted here, we can assert that under the construction of this discrete operator, the same result as in the continuous problem is achieved, as stated in Theorem \eqref{lemma:T-coer}.}    
\end{remark}

Finally, the injectivity of $\mathcal{A}_h$ is proven in the following lemma.

\begin{lemma}[Injectivity of $\mathcal{A}_h$]\label{lemma:inj-As}
The operator $\mathcal{A}_h$ is injective.
\end{lemma}

\begin{proof}
    Let $\vveh \in \UU_h$ such that $\mathcal{A}_h(\vveh)=0$. Then,
    \begin{align*}
    0=2 \vert \langle \mathcal{A}_h(\vveh), \vveh \rangle \vert \geq &  \re \langle \mathcal{A}_h(\vveh), \vveh \rangle +\im \langle \mathcal{A}_h(\vveh), \vveh \rangle \\
    = & \langle \mathbb{A}_1(\vvh), \vvh \rangle + \re \langle \mathbb{A}_2(\qh), \qh \rangle + \re \langle \tilde{\mathbb{D}}(\xih), \xih \rangle + \lambda^{-1}\bigg(\im (\xi,q)-\re(\xi,q)\bigg) +\theta \lambda^{-1}\Vert \qh\Vert^2_0 \\
    &- \lambda^{-1}\bigg(\im (\xi,q)+\re(\xi,q)\bigg) + \re \langle \mathcal{S}(\vveh), \vveh \rangle \\
    \geq & \alpha_{min} 2\mu_e \Vert \eps(\vvh)\vert^2_0+\frac{\kappa}{\alpha\,\omega\,\mu_f} \Vert \nabla \qh\Vert^2_0+\lambda^{-1}\Vert \xih \Vert^2_0 -\lambda^{-1}\Vert \xih\Vert^2_0 -\lambda^{-1}\Vert \qh\Vert^2_0 + \frac{\Se}{\alpha} \Vert \qh \Vert^2_0+\lambda^{-1}\Vert \qh\Vert^2_0 \\
    & \delta_1 \sumkth h^2_T \Vert \mathbf{R}(\vvh,\xih) \Vert^2_{0,T} +  \delta_2 \sumkth  \frac{h_T^2}{\mu_f\alpha\omega}\Vert \nabla \qh \Vert^2_{0,T} \\
    \geq & \alpha_{min} 2\mu_e \Vert \eps(\vvh)\vert^2_0+\frac{\kappa}{\alpha\,\omega\,\mu_f} \Vert \nabla \qh\Vert^2_0+\delta_1 \sumkth h^2_T \Vert \mathbf{R}(\vvh,\xih) \Vert^2_{0,T} +  \delta_2 \sumkth  \frac{h^2_T}{\mu_f\alpha\omega}\Vert \nabla \qh \Vert^2_{0,T},
    \end{align*}
    and then $\vvh = \zero$, $\qh =0$ and $\mathbf{R}(\vvh,\xih)= 0$, and in consequence $\xih=0$.
\end{proof}

To conclude this section, the following result establishes the existence and uniqueness of the solution for the discrete problem.

\begin{theorem}[Well-posedness of the discrete problem] \label{thm:inf-sup-dis}
     There exists $h_0>0$ such that for all $h>h_0$ the discrete problem \eqref{stab} has unique solution $\uve^{h*}\in\UU_h$. Moreover, there exists a positive constant $C_4$ independent of $h$ such that
    \begin{equation}
        \Vert \uve^{h*} \Vert_{\UU_h} \leq C_4 \Vert \mathcal{F}\Vert_{\UU_h}\leq C_4 \bigg( \Vert g^p \Vert_{0,\Gamma_{\uu}} +  \Vert \mathbf{g}_{re} \Vert_{0,\Gamma_{p}}\bigg) .
    \end{equation}
    or equivalently, there exists a positive constant $\beta_3>0$ such that 
    \begin{equation}
        \beta_3 \Vert \uveh\Vert_{\UU_h} \leq  \sup_{\vveh \in \UU_h \atop \vveh \neq \vec{\zero}} \frac{\langle \mathcal{A}_h(\uveh),\vveh\rangle }{\Vert \vveh\Vert_{\UU_h}}.
    \end{equation}
       
\end{theorem}
\begin{proof}
    The proof follows from the application of Lemmas \eqref{lemma:compact}, \eqref{lemma:coer_Bs} and \eqref{lemma:inj-As} plus the Fredholhm alternative.    
\end{proof}

\subsection{Convergence}
This section is dedicated to the convergence properties of the method. To this purpose, we assume additional regularity of the solution, i.e., considering the space 
\[
\WW = H^{k+1}(\OO)^d\times H^{k+1}(\OO) \times H^{k}(\OO)\,,
\] 
where $k$ is the order of the finite element spaces \eqref{eq:spaces-h}, and the Lagrange interpolation operators 
\begin{equation}
\begin{aligned}
& \IHh:\HH\cap H^{r+1}(\OO)^d\rightarrow \HH_h \\
& \IPh:P\cap H^{r+1}(\OO)\rightarrow P_h \\
& \ISh:S\cap H^{r}(\OO)\rightarrow S_h\,.
\end{aligned}
\end{equation}

Then, there exists a constant $\cla>0$ such that, 
\begin{equation}\label{interp}
\begin{split}
\Vert\vv-\IHh\vv\Vert_0+h\vert \vv-\IHh\vv\vert_1 \leq &\,\,\cla h^{k+1} \vert \vv\vert_{k+1}, \quad\forall  \vv \in \HH\cap H^{r+1}(\OO)^d\\
\Vert q-\IPh q\Vert_0+h\vert q-\IPh q\vert_1 \leq &\,\, \cla h^{k+1} \vert q\vert_{k+1}, \quad\forall  q\in P\cap H^{r+1}(\OO)\\
\Vert \xi-\ISh \xi\Vert_0\leq &\,\, \cla h^{k} \vert \xi\vert_{k}, \quad\forall  \xi\in S\cap H^{r}(\OO),
\end{split}
\end{equation}
for all $1\leq k\leq r$ (see, e.g., \cite[Theorem 1.103]{ERGU04}).
Let us also define $\III := (\IHh,\IPh,\ISh)$.

The next result concerns the theoretical rate of convergence for the Galerkin scheme \eqref{weakF}.

\begin{theorem}\label{teo:well_posed}[Convergence]
Let   $\uve$ and $\uveh$ be the solutions of \eqref{weakF} and \eqref{stab} respectively.  In addition, assume that $\uve \in \UU\cap \WW$. Then, for $1\leq k\leq r$, there exist two constants $C_1,C_2>0$, independent of $h$, such that,
\begin{equation}\label{eq:error-estimate}
	\Vert \uve -\uveh \Vert_{\UU} \leq\,\, C_1\,h^k\,\bigg(\Vert \uu \Vert_{k+1}+\Vert p \Vert_{k+1}+\Vert \phi \Vert_{k}\bigg) + C_2\delta_2\,h^2\,\Vert \nabla p\Vert_0
	\end{equation}
\end{theorem}
\begin{proof}
Hereafter, for $(\uu,p,\phi)\in(\HH, P,S )$ and for $(\uh,\ph,\phih)\in(\HH_h, P_h,S_h )$, let us introduce the notations	
\[
E(\uve) :=  \uve - \III\uve = 
\left(\uu-\IHh\uu,p-\IPh p,\phi-\ISh\xi\right)
\]
and
\[
E_h(\uve,\uve^h) :=  \uve^h - \III\uve
= \left(\uu^h-\IHh\uu,p^h-\IPh p,\phi^h-\ISh\xi\right)\,.
\]

% \begin{equation*}
% 		\begin{split} 
% 			&E(\uu)=\uu-\IHh\uu  \qquad \,\,\,\,\,E(p)=p-\IPh p  \qquad \,\,\,\,E(\phi)=\phi-\ISh  \phi   \\
% 			&E_h(\uu)=\uu^h-\IHh\uu  \qquad E_h(p)=p^h-\IPh p  \qquad E_h(\phi)=\phi^h-\ISh \phi.
% 		\end{split}
% 	\end{equation*}

% 	Furthermore, 
% 	\begin{align*} 
% 		&E(\uve) =  \uve - \III\uve = (E(\uu),E(p),E(\phi)) \\
% 		&E_h(\uve) =  \uve^h - \III\uve = (E_h(\uu),E_h(p),E_h(\phi)),
% 	\end{align*} 
% 	where $\III := (\IHh,\IPh,\ISh)$.

% Finally, we are positioned to establish the main result of this section, that is, the theoretical rate of convergence for the Galerkin scheme \eqref{weakF}.  In this context, it becomes evident that optimal convergence is attained only when considering a sufficiently large value of $\kappa$. This presents no significant issue, as, in the majority of practical applications, only low-order polynomial spaces are typically considered.

Let $\vveh \in \UU_h$. Using Theorem \eqref{thm:inf-sup-dis}  and Lemma \eqref{eq:ort} we obtain
\[ 
\beta_3 \Vert E_h(\uve,\uve^h)\Vert 
\leq \sup_{\vveh \in \UU_h \atop \vveh \neq \vec{\zero}} \frac{\vert \langle \mathcal{A}_h(E_h(\uve,\uve^h)),\vveh \rangle \vert}{\Vert \vveh\Vert_{\UU_h}} 
\leq  \sup_{\vveh \in \UU_h \atop \vveh \neq \vec{\zero}} \frac{\left\vert \langle \mathcal{A}_h(E(\uve)),\vveh \rangle - \displaystyle\delta_2 \sumkth \frac{h_T^2}{\mu_f\, \alpha \omega} (\nabla p, \nabla \qh)_T \right\vert}{\Vert \vveh\Vert_{\UU_h}}.\]
Utilizing the continuity of $\mathcal{A}_h$ (inequality \eqref{contAS}) \rrevision{and the properties \eqref{interp}} we get
\[ 
\vert \langle \mathcal{A}_h(E(\uve)),\vveh \rangle \vert \leq \eta_4 \Vert E(\uve)\Vert_{\UU}\Vert \vveh\Vert_{\UU_h} \leq \cla \,h^k\,\bigg(\vert\uu\vert_{k+1}+\vert p\vert_{k+1}+\vert \phi \vert_{k} \bigg)\Vert \vveh\Vert_{\UU_h}.\]
In addition, it holds
\begin{equation}
    \begin{aligned} 
    \bigg\vert \displaystyle\delta_2 \sumkth \frac{h^2_T}{\mu_f \alpha \omega}(\nabla p, \nabla \qh)_T \bigg\vert & 
    \leq  h^2 \delta_2 \frac{1}{\mu_f \alpha \omega} \bigg(\sumkth \Vert \nabla p\Vert^2_{0,T}\bigg)^{1/2}\bigg(\sumkth \Vert \nabla \qh\Vert^2_{0,T}\bigg)^{1/2} \\
 & = \delta_2 \bigg(\frac{h^2}{\kappa}\bigg)\,\Vert \nabla p\Vert_0 
 \left(\frac{\kappa}{\mu_f\alpha\omega}\right)^{\frac12} \Vert\vveh\Vert_{\UU_h} 
 %\Vert \qh \Vert_0 \\
 %& \leq  \tilde C_2\, \delta_2\, h^2\,\Vert \nabla p \Vert_0\, \Vert\vveh\Vert_{\UU_h}
 \end{aligned}
 \end{equation}
 with $\tilde C_2 = \left({\mu_f\alpha}{\kappa}\right)^{-\frac12}$,
 which allows to conclude
  \[ 
  \Vert E_h(\uve,\uve^h)\Vert \leq \beta_3^{-1} \cla\,h^k\,\bigg(\vert\uu\vert_{k+1}+\vert p\vert_{k+1}+\vert \phi \vert_{k} \bigg)+ \beta_3^{-1} \tilde C_2\, \delta_2\, h^2\,\Vert \nabla p \Vert_0\,.
  \]

Applying the triangle inequality and using 
\eqref{interp} yields
\begin{align*} \Vert \uve-\uveh \Vert_{\UU_h} & \leq \Vert E(\uve) \Vert_{\UU_h}+ \Vert E_h(\uve,\uve^h) \Vert_{\UU_h} \\
& \leq (1+\beta_3^{-1}) \cla \,h^k\,\bigg(\vert\uu\vert_{k+1}+\vert p\vert_{k+1}+\vert \phi \vert_{k} \bigg)+ \beta_3^{-1} \tilde C_2\, \delta_2\, h^2\,\Vert \nabla p \Vert_0 \end{align*}
concluding the proof.
 \end{proof}

\section{Numerical Examples}\label{sec:results}
This section is devoted to the numerical results.
The first three examples aim at validating the method and the theoretical expectations
presented in Section \ref{sec:conclusions}. For these purposes, we introduce some analytical solutions. Since every solution $v$ is a complex function, they are written as $v = (\re v \,\, \im v)$.  
Additional examples will address the 
robustness of the solver in layered domains as well as its application 
in a realistic setting using a brain geometry segmented from magnetic resonance medical images.

% put this in each example - related to the used solver
% Concerning the computational aspect, the software MAD (\cite{galarceThesis}, chapter 5) is used for the finite element framework, based upon the linear algebra library PETSc \cite{petsc}. The inversion of the system of equations is done by means of the MUltifrontal Massively Parallel sparse direct Solver (MUMPS, \cite{MUMPS}) for the example 1, and the generalized minimal residual method (GMRES) iterative solver for examples 2 and 3. \rrevision{To check the in-house finite element implementation, the convergence test in example 1 was also performed using the finite the open source library FEniCS \cite{LoWeMa11}.}

\subsection{Example 1: Validation against an analytical solution}\label{example1}

\begin{figure}[!htbp]
\centering
	 \subfigure[$\re \uu^h$]{
            \includegraphics[height = 2.5cm]{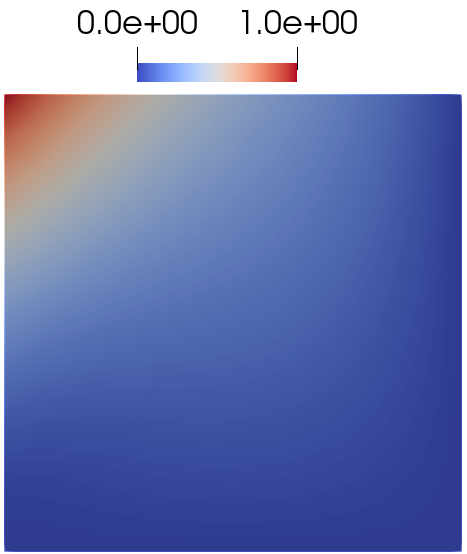}}
	\subfigure[$\im \uu^h$]{
		\includegraphics[height = 2.5cm]{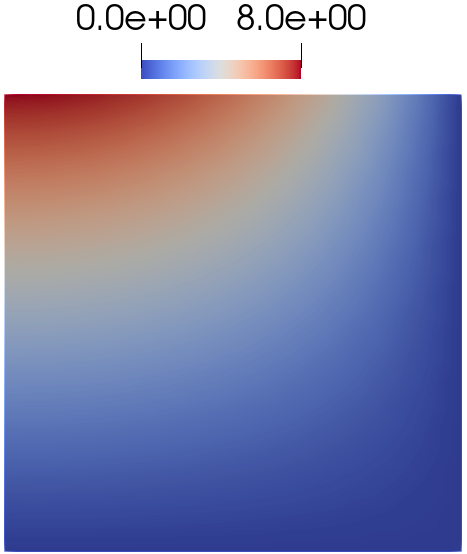}}
	\subfigure[$\re p^h$]{
		\includegraphics[height = 2.5cm]{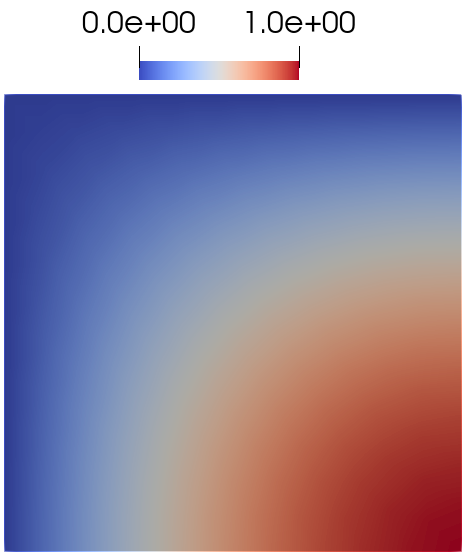}}
	\subfigure[$\im p^h$]{
		\includegraphics[height = 2.5cm]{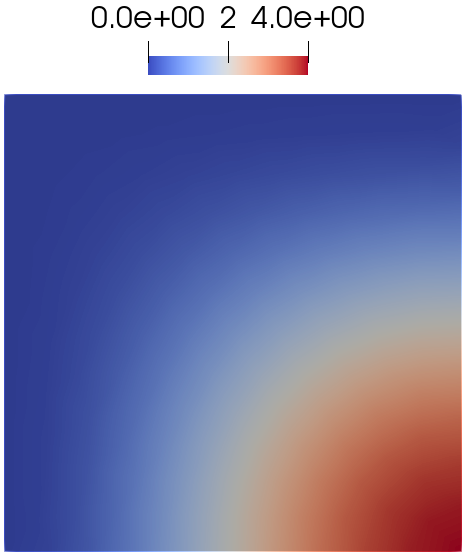}}
	\subfigure[$\re \phi^h$]{
		\includegraphics[height = 2.5cm]{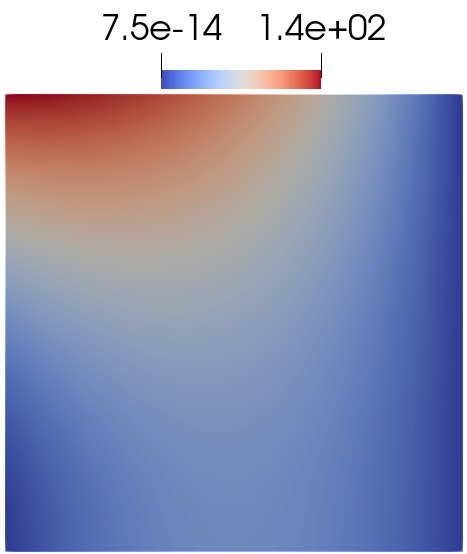}}
	\subfigure[$\im \phi^h$]{
	 	\includegraphics[height = 2.5cm]{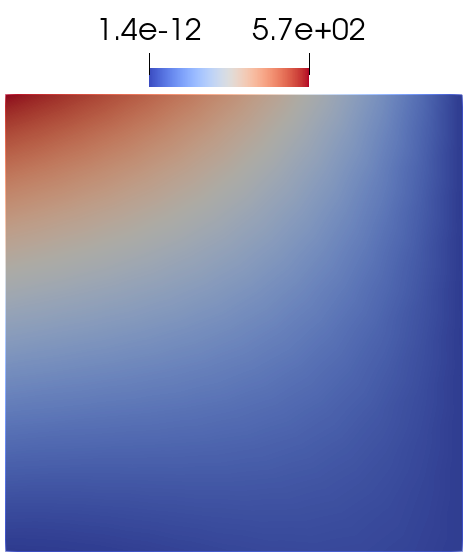}} \\
	%  \subfigure[$\ure^h$]{
 %              % Answer: [trim={left bottom right top},clip]
	% 	\includegraphics[height = 2.2cm, trim={0.4cm 0 0.5 1.8cm},clip]{./fig/Example2D/ure_fem.png}}
	% \subfigure[$\uim^h,$]{
	% 	\includegraphics[height = 2.2cm, trim={0.4cm 0 0.5 1.8cm},clip]{./fig/Example2D/uim_fem.png}}
	% \subfigure[$\pre^h$]{
	% 	\includegraphics[height = 2.2cm, trim={0.4cm 0 0.5 1.8cm},clip]{./fig/Example2D/pre_fem.png}}
	% \subfigure[$\pim^h$]{
	% 	\includegraphics[height = 2.2cm, trim={0.4cm 0 0.5 1.8cm},clip]{./fig/Example2D/pim_fem.png}}
	% \subfigure[$\phre^h$]{
	% 	\includegraphics[height = 2.2cm, trim={0.4cm 0 0.5 1.8cm},clip]{./fig/Example2D/phire_fem.png}}
	% \subfigure[$\phim^h$]{
	% 	\includegraphics[height = 2.2cm, trim={0.4cm 0 0.5 1.8cm},clip]{./fig/Example2D/phiim_fem.png}}
	\centering
	\caption{Example 1 (2D): \revision{Visualization of the analytical solution.}}
    \label{fig:test0_results}
\end{figure}

We first validate the numerical method in a case in which the problem \eqref{eq:biot-fft} can be solved analytically, and whose solution is given by,

\begin{equation}\label{eq:analytical-2d}
\begin{aligned}
\uu = & \begin{pmatrix}
 \re \uu & \im \uu
\end{pmatrix} = \begin{pmatrix}
 (x-1)^2y^2 & (x-1)(x+2)^2y(y+1) \\ xy(x-1)    &2x^2y(x-1)   
\end{pmatrix}\\
p =& \begin{pmatrix}
 \re p & \im p
\end{pmatrix} = \begin{pmatrix}
    \sin\bigg(\frac{\pi x}{2}\bigg)\cos\bigg(\frac{\pi y}{2}\bigg) & (1-\cos(\pi x))(1+\cos(\pi y))
\end{pmatrix}\\
\phi =& \begin{pmatrix}
 \re \phi & \im \phi
\end{pmatrix} = \begin{pmatrix}
    \sin\bigg(\frac{\pi x}{2}\bigg)\cos\bigg(\frac{\pi y}{2}\bigg)-\lambda \tr(\eps(\re\uu)) & (1-\cos(\pi x))(1+\cos(\pi y))- \lambda \tr(\eps(\im\uu))
\end{pmatrix}
\end{aligned}
\end{equation}
for 2D on the unit square and,
\begin{equation}\label{eq:analytical-3d}
\begin{aligned}
\uu = & \begin{pmatrix}
 \re \uu & \im \uu
\end{pmatrix} = \begin{pmatrix}
 (x-1)^2xy^2(z+2)  &  (x-1)x^2y^2(z+1) \\ x^3y(z+1)^2(x-1)^3    &(x-1)^2x^3y^3(z-2) \\  (x-1)xy(z+2)^2 & x^2y(x-1)(z-2)^2   
\end{pmatrix}\\
p =& \begin{pmatrix}
 \re p & \im p
\end{pmatrix} =\begin{pmatrix}
    \sin\bigg(\frac{\pi x}{2}\bigg)\cos\bigg(\frac{\pi y}{2}\bigg) & (1-\cos(\pi x))(1+\cos(\pi y))
\end{pmatrix}\\
\phi =& \begin{pmatrix}
 \re \phi & \im \phi
\end{pmatrix} = \begin{pmatrix}
    \cos\bigg(\frac{\pi y}{2}\bigg)\sin(\pi z)-\lambda\tr \eps(\re \uu)& 
    \sin\bigg(\frac{\pi}{2} z\bigg)(1+\cos(\pi y))(1+\cos(\pi z))-\lambda\tr \eps(\im\uu)
\end{pmatrix}
\end{aligned}
\end{equation}
for 3D on the unit cube, respectively.

Boundary conditions are prescribed according to the exact solutions \eqref{eq:analytical-2d} and \eqref{eq:analytical-3d}, evaluated on the sets:
$$
\Gamma_p^{2D} = \{(x,y)\in\RR^2: x =0,\,y=1\}, 
$$
and,
$$
\Gamma_u^{2D} =\{(x,y)\in\RR^2: x =1,\,y=0\},
$$  
for pressure and velocity in 2D (resp.), and on,
$$
\Gamma_p^{3D} = \{(x,y,z)\in\RR^3: x =0,\,x=1,\, y=0\},
$$ 
and,
$$
\Gamma_u^{3D} = \{(x,y,z)\in\RR^3: y =1,\,z=0,\, z=1\},
$$ 
for pressure and velocity in 3D, respectively. 

The set of physical parameters for both 2D and 3D simulations is described in table \eqref{tab:parameters}.

\begin{table}[!htbp]
\centering
\begin{tabular}{@{}lllllll@{}}
\toprule
Parameter & $E$ [dyn/cm$^2$] & $\nu$ & $\mu_f$ [Poise] & $\kappa$ [cm$^2$] & $\omega$ & $\rho$ [gr/$cm^3$] \\ \midrule
Value     & $10^{2}$ & 0.4   & 1.0    & $0.1$     & 1.0      & 1.0     \\ 
\bottomrule
\end{tabular}
\caption{Physical parameters used for Example 1.}
\label{tab:parameters}
\end{table}

The computational domain is based on several refinements of a unstructured triangular/tetrahedron mesh (coarsest \revision{discretization $h=0.5$ cm}). 
\rrevision{The finite element formulation for this example has been implemented using the library FEniCS \cite{logg2012fenics}, and the solution of the linear systems is based on the direct solver MUMPS (MUltifrontal Massively Parallel sparse direct Solver, \cite{MUMPS}).}

The numerical solutions for the two-dimensional version of this problem are shown in Figure \ref{fig:test0_results}, together with the exact solutions \eqref{eq:analytical-2d}.
Figure \ref{fig:convergence} shows the error with respect to the exact solution for displacement, pressure, and total pressure, as a function of the mesh size, setting 
$\delta_1=0.5$ and $\delta_2=0.0$. The obtained convergence rates confirm the theoretical expectations 
discussed in Section \ref{sec:discrete-analysis} both for linear and quadratic finite elements.

\begin{figure}[!h]
\centering
\includegraphics[width=\textwidth]{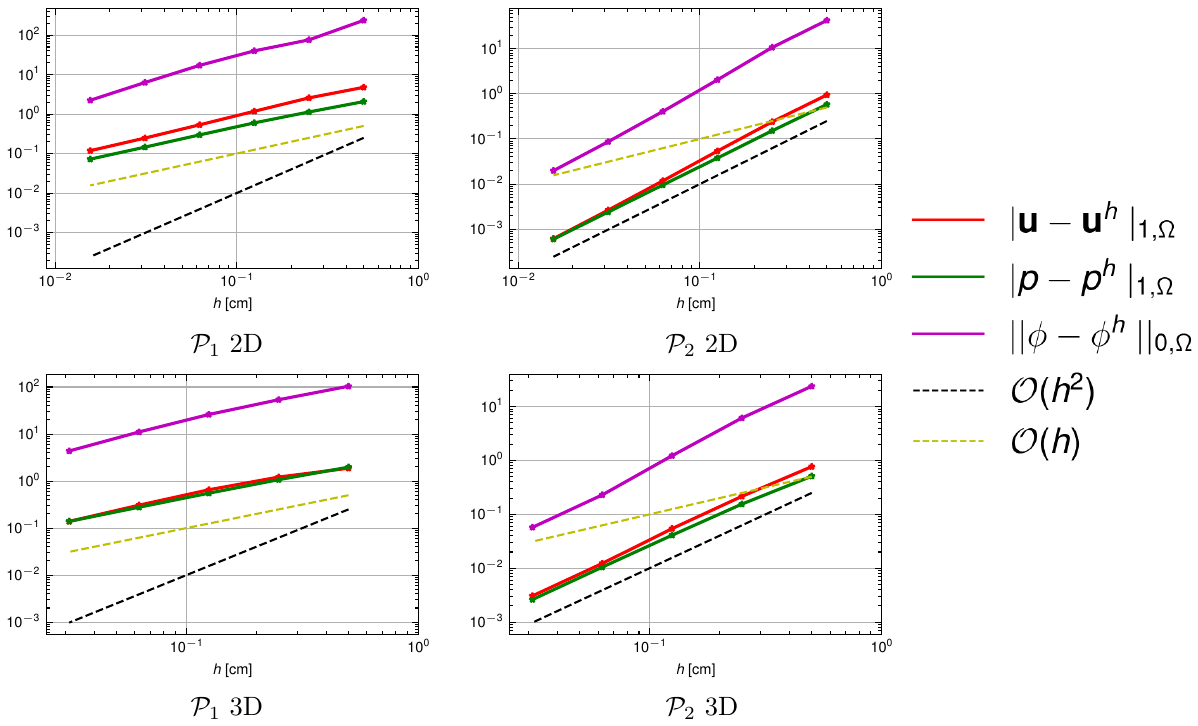}
\caption{Example 1. Error for displacement (red), pressure (green), and total pressure (magenta) as a function of the mesh size \revision{in centimeters}. Top row: two-dimensional problem, bottom row: three-dimensional problem. Left: linear finite elements, right: quadratic finite elements. The dashed lines in each plot show the convergence rate $O(h)$ and $O(h^2)$.}
\label{fig:convergence}
\end{figure}

\rrevision{As next, we investigate the
effect of including the additional pressure stabilization ($\delta_2>0)$ which, although not
required for the inf-sup stability and for
obtaining the expected convergence rates, is expected to improve the performance of the solver
for small permeabilities. To demonstrate numerically the relevance of this term, we performed 
computations varying the physical parameters ($\kappa$ and $\nu$), the mesh size, the order of the finite elements, and the parameter $\delta_2$.
Figure \ref{fig:stab_kappa} shows the behavior of the numerical error decreasing progressively the permeability $\kappa$. Setting $\delta_2=0$ yields a significant increase in the error for $\kappa \ll 1$, unless the discretization
is refined below a certain threshold (in the case of the example, stable results for the smallest $\kappa$
are obtained for second order elements and mesh size $h=0.015625$, Figure \ref{fig:stab_kappa}, bottom-right).

However, the stabilized version leads to 
errors almost independent on the permeability also for coarser meshes and linear elements. 
Figure \ref{fig:stab_kappa} shows the results for 
$\delta_2 \in \{0.0001,0.01,1\}$. Our numerical experiments shows, moreover, that even smaller value of $\delta_2$ are enough, in this example, to overcome the accuracy issues. The obtained errors are similar for 
$0< \delta_2 \leq 1$.

Figure \ref{fig:stab_nu} shows the behavior of the error varying the Poisson modulus $\nu$ and for two $\delta_2 \in \{0,1\}$. In this case, one cannot 
observe any sensible difference among the different setups.
}

\begin{figure}[!htbp]
\centering
\begin{tabular}{ccc}
$h=0.0625$cm & $h=0.03125$cm & $h=0.0015625$cm \\[0.2em]
\includegraphics[width=5cm,trim=0cm 0cm 7.3cm 0cm,clip=true]{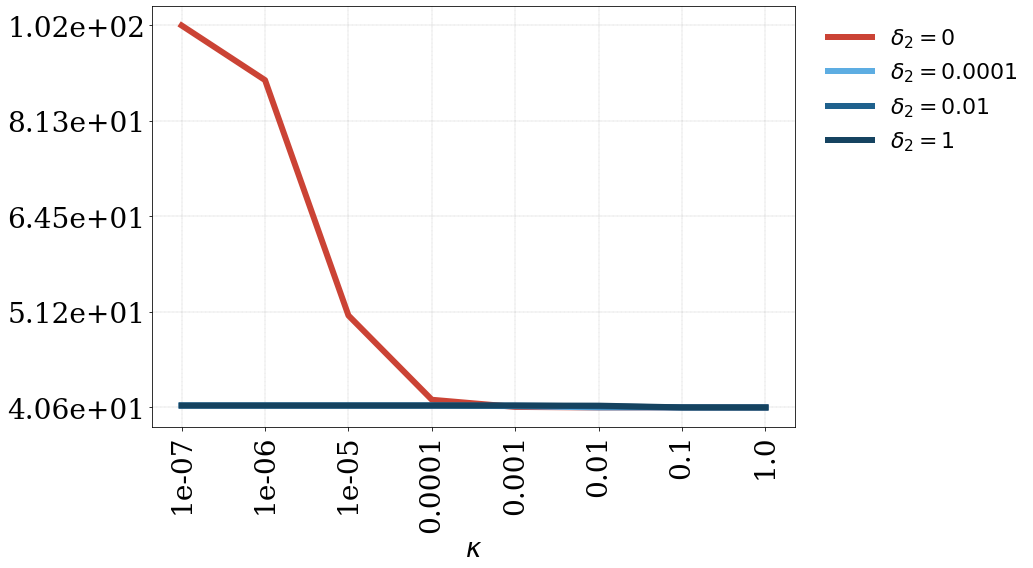} & 
\includegraphics[width=5cm,trim=0cm 0cm 7.3cm 0cm,clip=true]{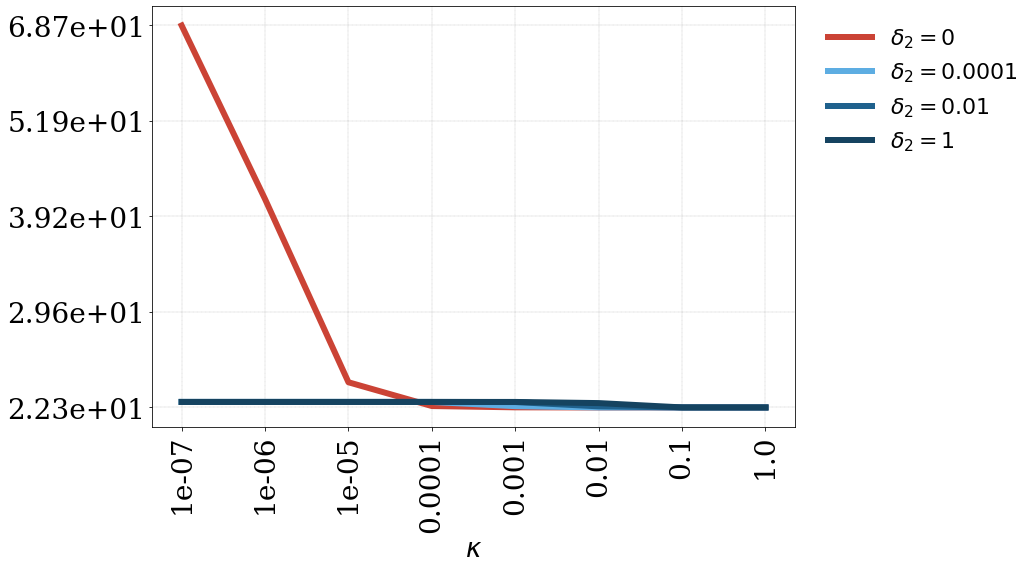} &
\includegraphics[width=5cm,trim=0cm 0cm 7.3cm 0cm,clip=true]{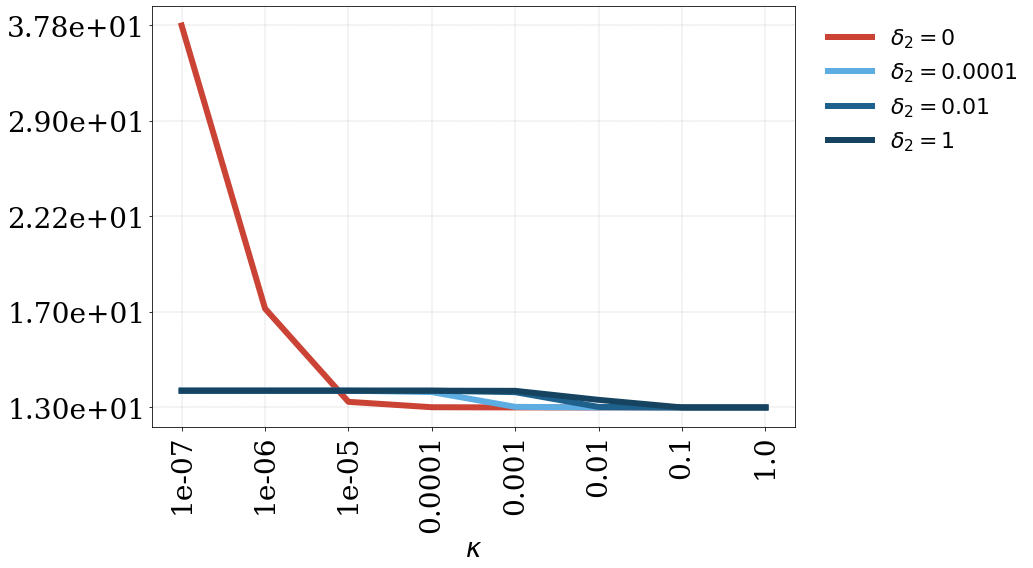} \\[0.2em]
\includegraphics[width=5cm,trim=0cm 0cm 7.3cm 0cm,clip=true]{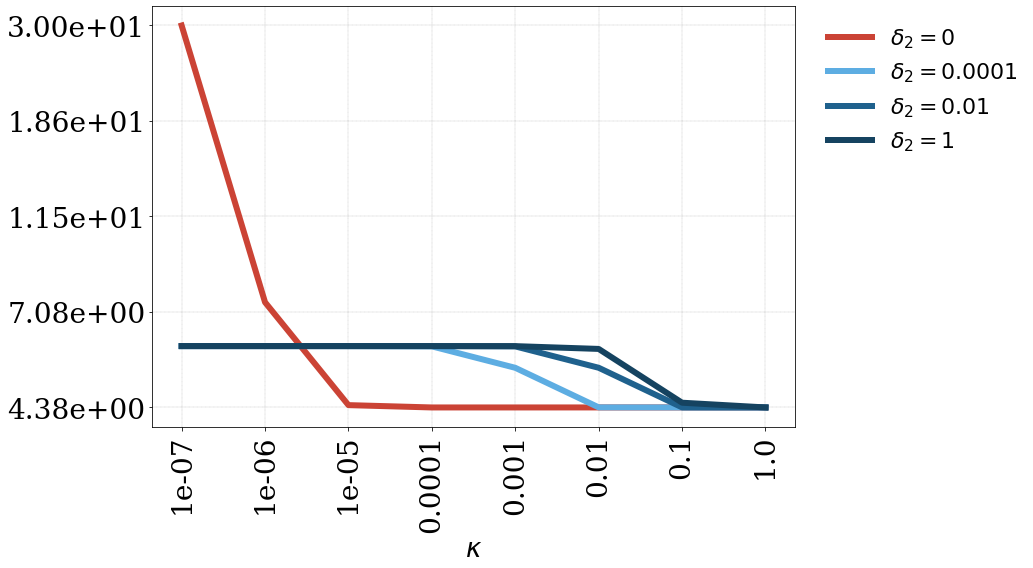} &
\includegraphics[width=5cm,trim=0cm 0cm 7.3cm 0cm,clip=true]{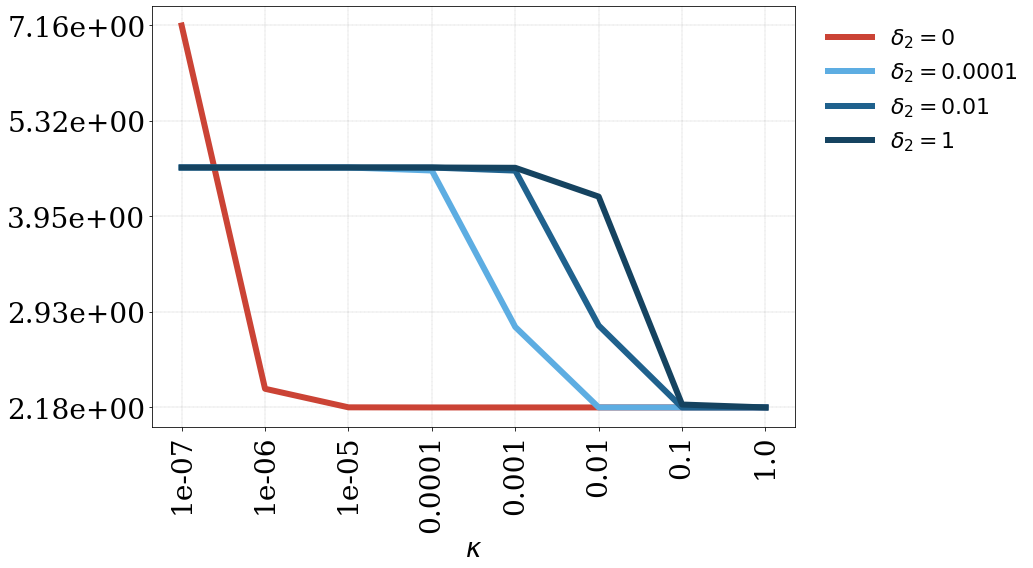} &
\includegraphics[width=5cm,trim=0cm 0cm 7.3cm 0cm,clip=true]{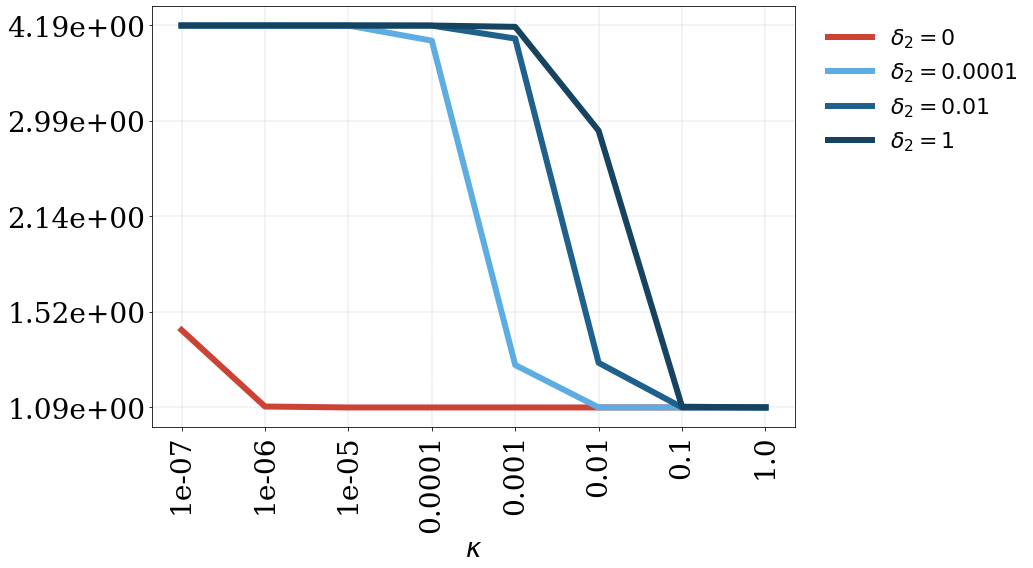}
\end{tabular}

% legend
\includegraphics[width=2cm,trim=0cm 0cm 0cm 0cm,clip=true]{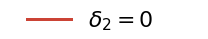}
\includegraphics[width=2cm,trim=0cm 0cm 0cm 0cm,clip=true]{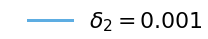}
\includegraphics[width=2cm,trim=0cm 0cm 0cm 0cm,clip=true]{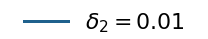}
\includegraphics[width=2cm,trim=0cm 0cm 0cm 0cm,clip=true]{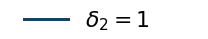}
	\centering
	\caption{\rrevision{Example 1 (2D). Numerical error in the norm \eqref{norm:cont} for different values of the permeability, for decreasing mesh sizes (from left to right), and for $\delta_2 = 0$ (red), and $\delta_2 \in \{0.0001, 0.01, 0.1\}$ (lighter to darker blue). 
 Top row: $\mathcal{P}_1$ elements, bottom row:
 $\mathcal{P}_2$ elements.} }
    \label{fig:stab_kappa}
\end{figure}

\begin{figure}[!ht]
\centering
\begin{tabular}{ccc}
$h=0.0625$cm & $h=0.03125$cm & $h=0.0015625$cm \\[0.2em]
\includegraphics[width=5cm,trim=0cm 0cm 7.3cm 0cm,clip=true]{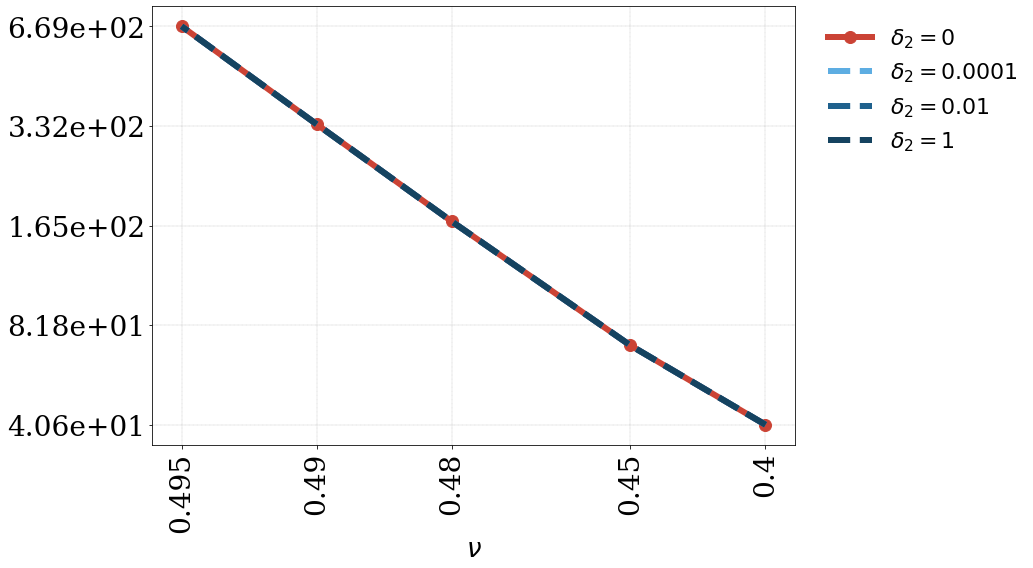} & 
\includegraphics[width=5cm,trim=0cm 0cm 7.3cm 0cm,clip=true]{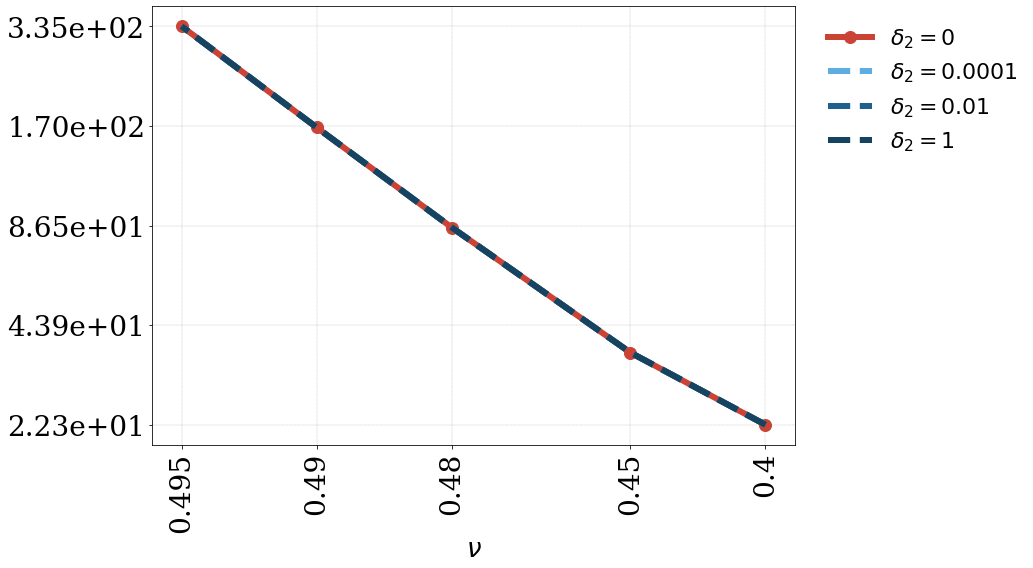} &
\includegraphics[width=5cm,trim=0cm 0cm 7.3cm 0cm,clip=true]{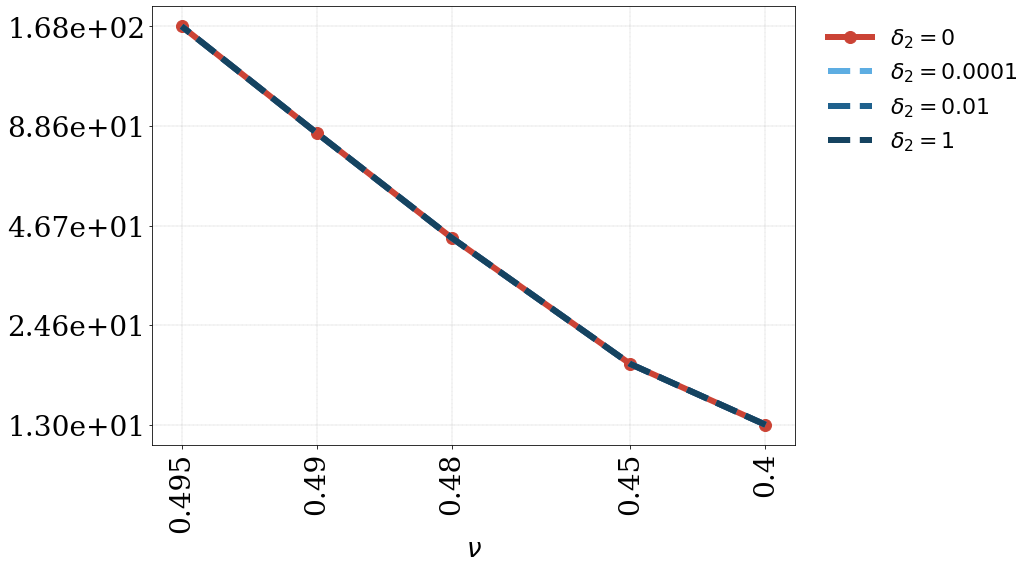} \\[0.2em]
\includegraphics[width=5cm,trim=0cm 0cm 7.3cm 0cm,clip=true]{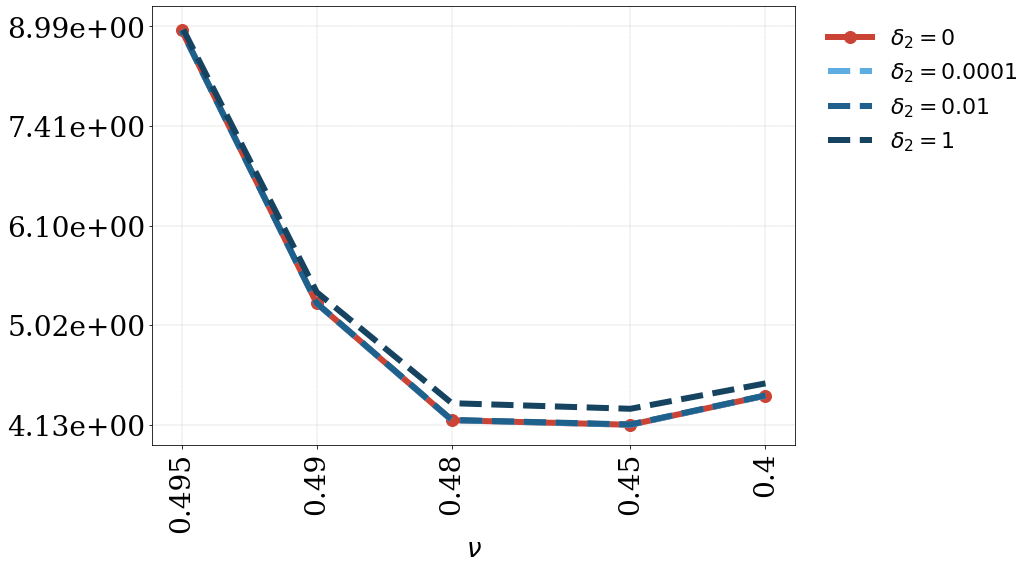} &
\includegraphics[width=5cm,trim=0cm 0cm 7.3cm 0cm,clip=true]{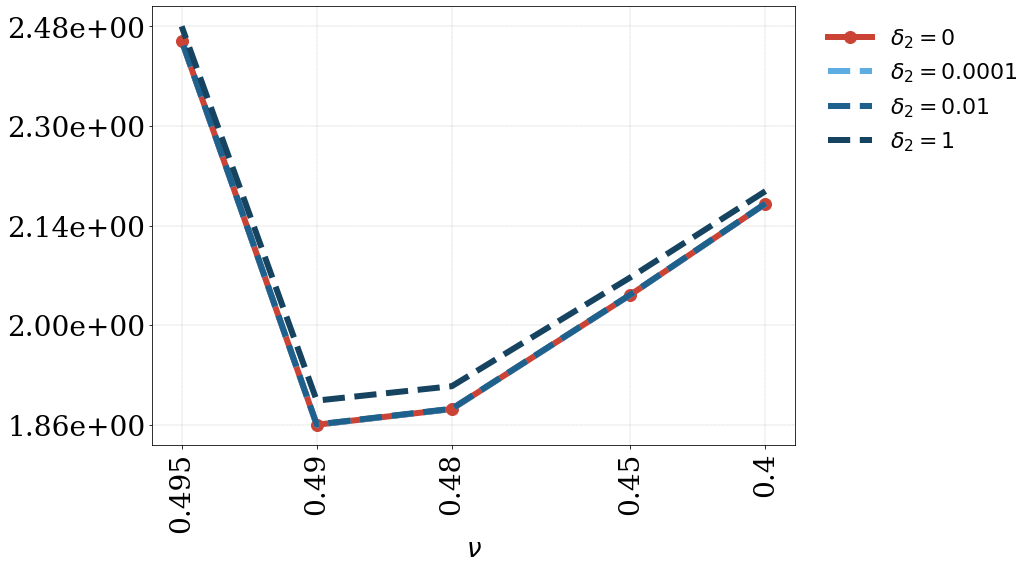} &
\includegraphics[width=5cm,trim=0cm 0cm 7.3cm 0cm,clip=true]{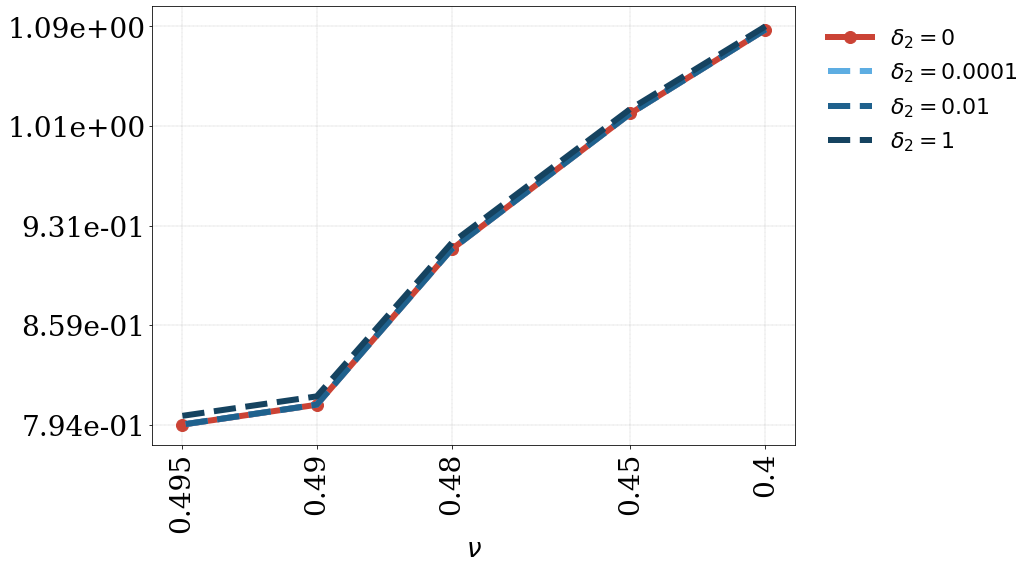}
\end{tabular}

% legend
\includegraphics[width=2cm,trim=0cm 0cm 0cm 0cm,clip=true]{./figures/example_1/d2_0.png}
\includegraphics[width=2cm,trim=0cm 0cm 0cm 0cm,clip=true]{./figures/example_1/d2_1em4.png}
\includegraphics[width=2cm,trim=0cm 0cm 0cm 0cm,clip=true]{./figures/example_1/d2_1em2.png}
\includegraphics[width=2cm,trim=0cm 0cm 0cm 0cm,clip=true]{./figures/example_1/d2_1.png}
\caption{\rrevision{Example 1 (2D): Numerical error in the norm \eqref{norm:cont} for different values of the Poisson modulus $\nu$, for decreasing mesh sizes (from left to right), and for $\delta_2 = 0$ (red), and $\delta_2 \in \{0.0001, 0.01, 0.1\}$ (lighter to darker blue, dashed lines). 
  Top row: $\mathcal{P}_1$ elements, bottom row:
 $\mathcal{P}_2$ elements.}  
}
\label{fig:stab_nu}
\end{figure}

\subsection{Example 2: Layered domain}\label{example2}
In this section we considered a two-dimensional domain containing layers with different permeabilities. The purpose of this example is to validate the robustness of the numerical solutions, in particular of the pressure,
in presence of discontinuous coefficients, spanning different orders of magnitude.
Addressing such problems is relevant in different fields of applications, including soil mechanics and biomedical engineering, particularly in scenarios where the system parameters are affected by uncertainty and/or have to be estimated. 

We set $\Omega = [0,1] \times [0,1]$, decomposed in three subdomains with $\kappa_1 = 10^{-3}$ cm$^2$ for $y \in [0,1/3]$, $\kappa_2 = 10^{-4}$ cm$^2$, for $y \in [1/3, 2/3]$ and $\kappa_3 = 10^{-5}$ cm$^2$ for $y=[2/3,1]$.
(see Figure \ref{fig:layers_setup}). 
The values of the other physical parameters are provided in Table \ref{tab:parameters2}.

\begin{table}[!h]
\centering
\begin{tabular}{@{}lllllll@{}}
\toprule
Parameter & $E$ [dyn/cm$^2$] & $\nu$ & $\mu_f$ [Poise] & $\kappa$ [cm$^2$] & $\omega$ [Hz] & $\rho$ [gr/$cm^3$] \\ \midrule
Value & $10^{2}$ & 0.45 & $10^{-2}$ & $10^{-3}$ $|$ $10^{-4}$ $|$ $10^{-5}$ & 25 $|$ 50 $|$ 75 $|$ 100 $|$ 125 & 1.0 \\ \bottomrule
\end{tabular}
\caption{Parameters used for the Example \eqref{example2} with different permeabilites.}
\label{tab:parameters2}
\end{table}

\begin{figure}[!htbp]
    \centering
    \includegraphics[width=0.4\textwidth]{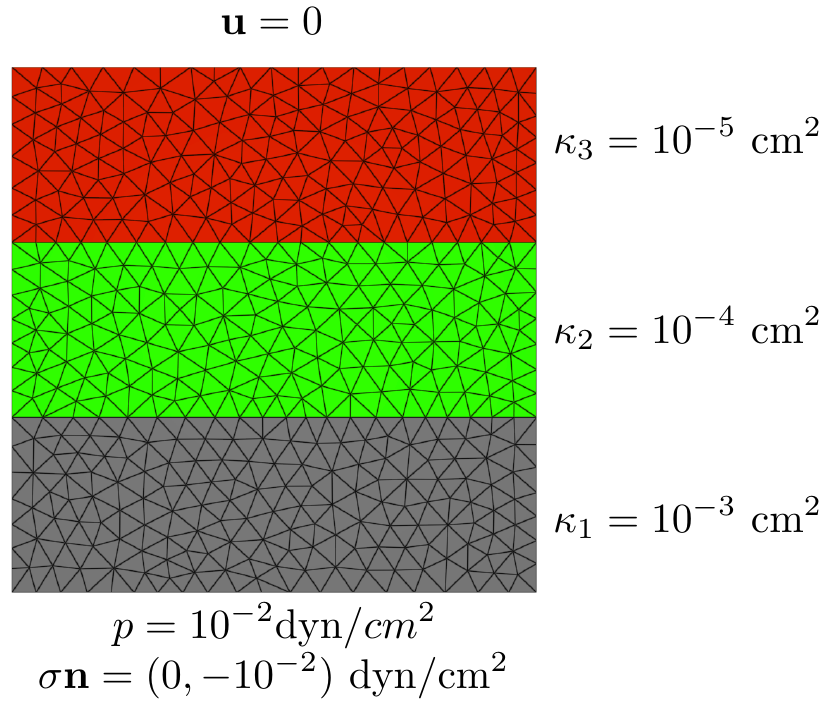}
    \caption{Example 2 (layered domain): Set-up of the computational domain with varying permeability (on a coarse mesh).}
    \label{fig:layers_setup}
\end{figure}

Concerning the boundary conditions, we set a Neumann boundary condition on the square bottom of magnitude $10^{-2}$ dyn/cm$^2$ pointing upwards, zero displacements at the top of the geometry, and a constant pressure field at the bottom of $10^{-2}$ dyn/cm$^2$. 
We consider an unstructured triangular mesh with characteristic size of $h=4 \times 10^{-3}$ cm, and stabilization parameters 
$\delta_1 = \omega^{-2}$ (inf-sup stabilization) and $\delta_2$ = 1 (pressure stabilization).
\rrevision{The numerical solution is obtained with the
library MAD \cite[Chapter 5]{galarceThesis}, and the finite element system is solved using the iterative method GMRES with an additive Schwarz preconditioner and employing a restart of the method every 500 iterations.}

The magnitude of the solutions obtained for different values of $\omega$ are shown in Figures \ref{fig:u_phi}, for displacement and total pressure.
\begin{figure}[!htbp]
    \centering
    \includegraphics[width=0.8\textwidth]{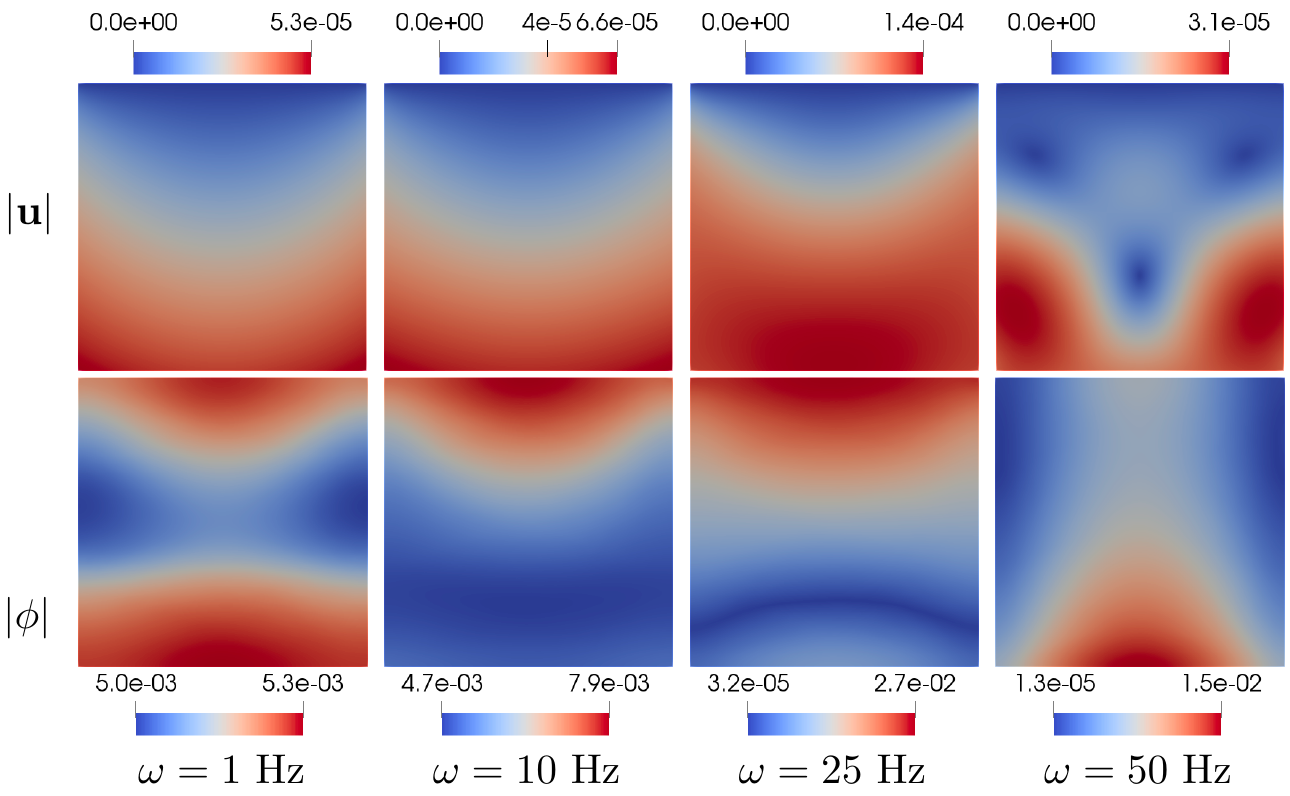}
    \caption{\rrevision{Example 2} (layered domain): Magnitudes of the displacement and total pressure solutions for different excitation frequencies. The stabilization parameters in \eqref{eq:S} are set to $\delta_1=\omega^{-2}$ and $\delta_2= 1.0$}
    \label{fig:u_phi}
\end{figure}
The solutions for the pressure are analyzed in more detail in Figure \ref{fig:p_profile}, highlighting that the numerical solution is not affected by the discontinuities and by the small values of the permeability.
\begin{figure}[!h]
    \centering
    \includegraphics[width=0.85\textwidth]{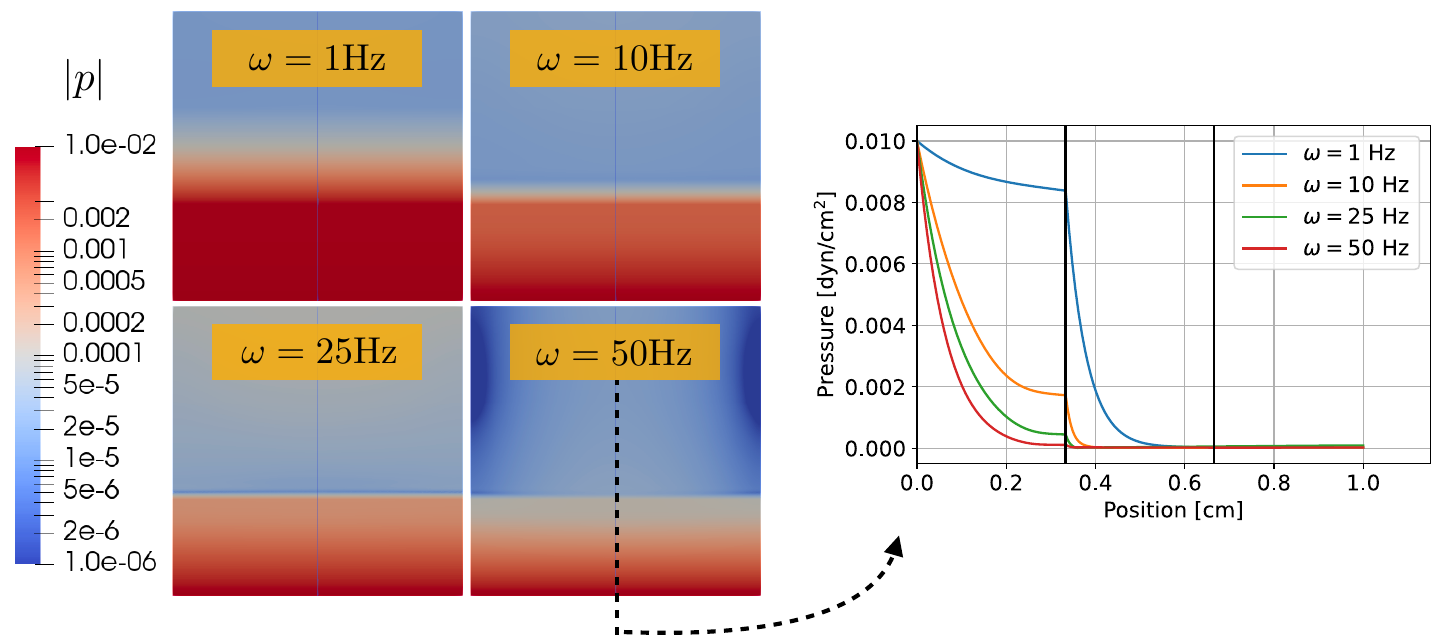}
    \caption{Example 2 (layered domain): Pressure profiles at vertical control line. The solver shows a robust behavior against discontinuities in the permeability field. The permeability interfaces are indicated with black lines in the right plot.}
    \label{fig:p_profile}
\end{figure}

Although an exact solution for this benchmark is not available, 
One can further infer the validity of the results in terms of the expected elastic behavior of the wave within the media
The parameter set of the simulations impose a wave speed $\sqrt{E/\rho} = 10$ cm/s, leading to wavelengths of approximately 62 cm, 6.28 cm, 2.51 cm, 1.27 cm for frequencies of 1 Hz, 10Hz, 25 Hz and 50 Hz, respectively. 
This explains why in the low frequency simulation the domain size (1 cm) does not allow a full wave cycle to develop, whereas an almost full wavelength is depicted for $\omega = 50 Hz$.

\subsection{Example 3: Three-dimensional brain geometry}\label{ssec:brain}
The final test case considers the simulation of a magnetic resonance elastography (MRE) experiment on a realistic brain geometry.
\rrevision{The goal of this example is to benchmark the performance of the numerical method in a realistic context, considering both a complex geometry and physiological parameters.}

\rrevision{The mesh has been obtained from MRI images (MPRAGE sequence) acquired at the Department of Radiology, Charité - Universitätsmedizin Berlin, Germany, and the resulting} computational domain is depicted in Figure \ref{fig:brain_geometry}. 
The boundaries of the domain are decomposed in disjoint sets $\partial \Omega = \Gamma_{\text{ventricles}} \cup \Gamma_{\text{outer}}$, with $\Gamma_{\text{outer}} = \partial \Omega $ \textbackslash  $\Gamma_{\text{ventricles}}$. 
\rrevision{The boundary conditions in the frequency domain are prescribed as follows (see also Figure \ref{fig:brain_geometry}).}
\begin{itemize}
\item \rrevision{A harmonic excitation at frequency $\omega$ of given magnitude on the displacement on $\Gamma_{\text{outer}}$, which attenuates with the x direction according to the function $s(x) = (1-x/16.86)$, i.e., 
\[
\sigma \textbf{n} = [0, s(x), 0],\;\mbox{on}\; \Gamma_{\text{outer}}.
\]
This result in a maximal pulse intensity on the brain back and a vanishing force on the front. }
The intensity of the MRE pulse boundary condition is chosen to get physiological magnitudes for the displacement fields \rrevision{(of the order of 10 $\mu$m)}. 
\item \rrevision{The intracranial and ventricle pressures 
are assumed to be constant in time during the pulse and not affected by the MRE excitation, yielding
\[
p=0,\;\mbox{on}\; \Gamma_{\text{outer}} \cup \Gamma_{\text{ventricles}}.
\]}
\item \rrevision{Homogeneous Dirichlet boundary condition is assumed on the brain ventricles, i.e
\[
\uu=0,\;\mbox{on}\; \Gamma_{\text{ventricles}.}
\]}
\end{itemize}
\begin{table}[!htbp]
\centering
\begin{tabular}{@{}lllllll@{}}
\toprule
Parameter & $E$ & $\nu$ & $\mu_f$ & $\kappa$ & $\omega$ & $\rho$ \\ \midrule
Value     & $10^4$ [dyn/cm$^2$] & 0.4   & 0.01 [Poise]   & $10^{-8}$ [cm$^2$]    & 10.0 [Hz]   & 1.0 [gr/$cm^3$]    \\ \bottomrule
\end{tabular}
\caption{Parameter set for human brain simulation.}
\label{tab:parameters_realBrain}
\end{table}

The physical parameters are chosen to mimic the setting presented in \cite{GCS-2023} (see Table \ref{tab:parameters_realBrain}). 

\begin{figure}[!htbp]
   \centering
   \includegraphics[height=5cm]{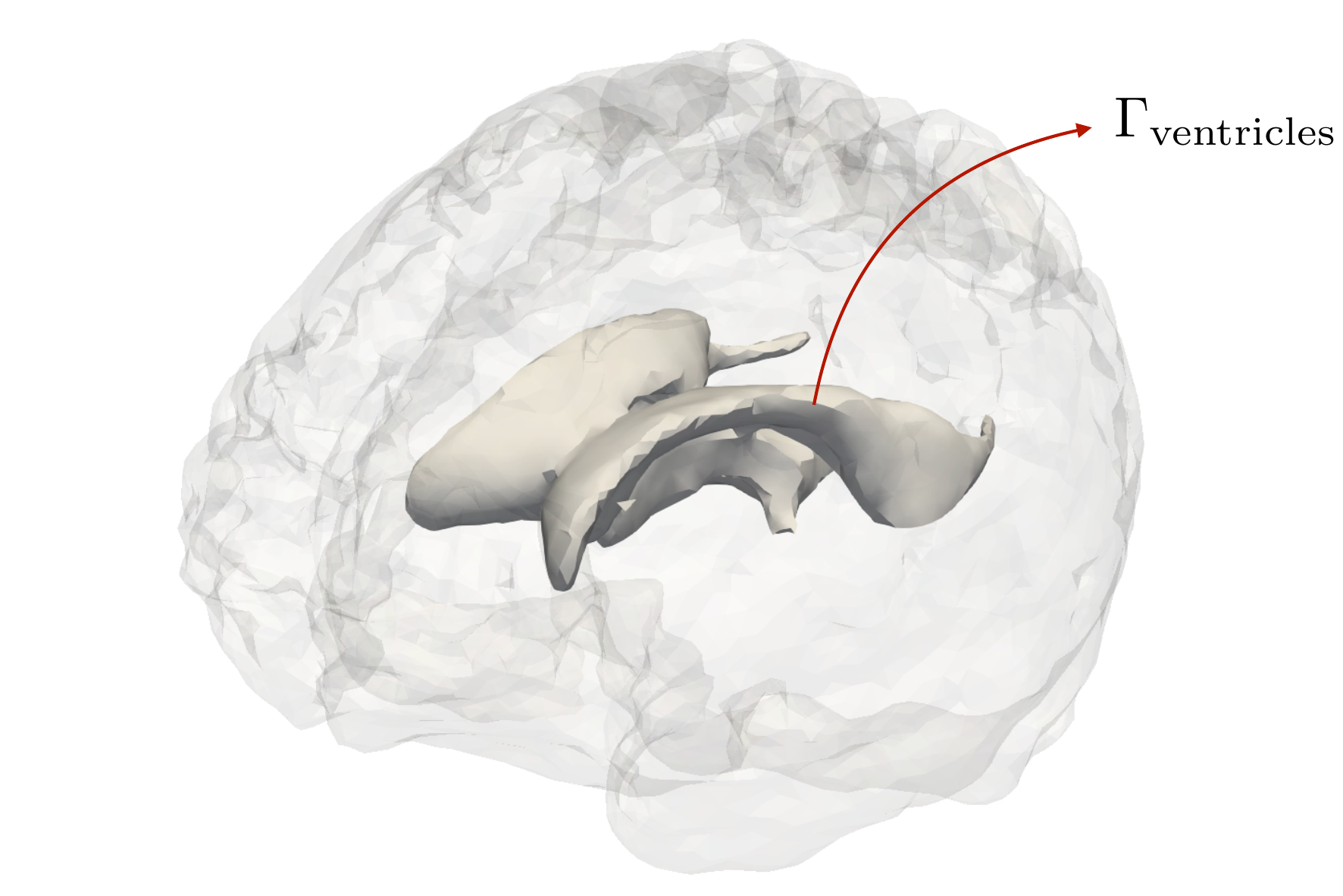}
\includegraphics[height=5cm]{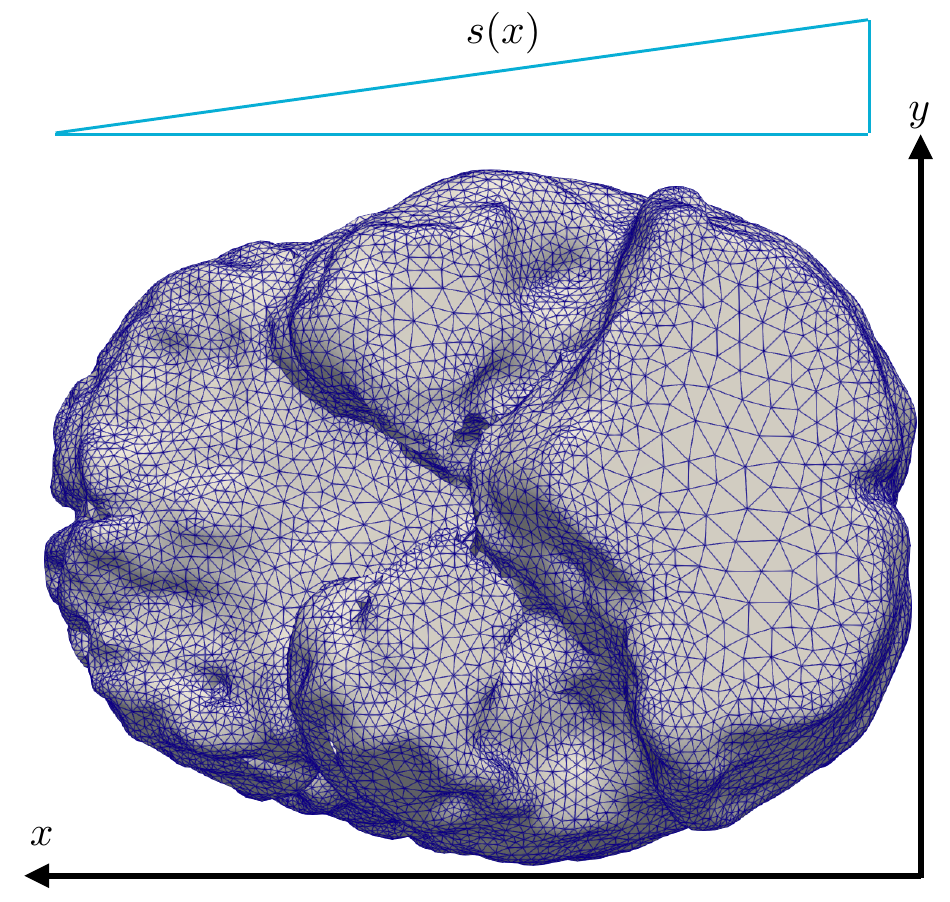}
   \caption{Example 3 (brain elastography): Computational domain (Source: \cite{GCS-2023})}
   \label{fig:brain_geometry}
\end{figure}

The volume discretization (tetrahedral mesh) has been generated from the segmented surface triangulation using the software MMG\cite{mmg}. It contains about 103,000 vertices and 511,000 tetrahedra. 
Employing a linear equal-order finite elements leads to a linear system with about $1.03 \times 10^6$ degrees of freedom. 
\rrevision{The numerical solution is obtained with the
library MAD \cite[Chapter 5]{galarceThesis}, and the finite element system is solved using the iterative method GMRES with an additive Schwarz preconditioner and employing a restart of the method every 500 iterations.}

\rrevision{The value of the permeability is very low
(Table \ref{tab:parameters_realBrain}). With $\delta_2=0$, the iterative solver does not converge, while we observe a better behavior including the pressure stabilization.
The numerical study performed with different values of $\delta_2 \in [0,10]$ shows that $\delta_2 = 0.01$ allows to obtain convergence and that the number of required GMRES iterations decrease increasing the stabilization parameter (see Table \ref{tab:ksp_brain}). On the other hand, 
high values of $\delta_2$ introduce additional permeability and yield a more diffuse pressure field (Figure \ref{fig:testBrain_results}), while perturbing less the displacement field. 
These results suggest that the
value of the stabiliation, in practice, shall be chosen to balance between better numerical behavior and overall accuracy. At the same time, we observe that the impact of the pressure stabilization on the conditioning of the finite element system might play a very important role in the context of inverse problems where the forward model typically needs to be 
solved multiple times and physical parameters (for example, the permeability) can vary  during the solution of the problem reaching critical values.
}
%{\color{red} We observe an adequate behavior of the solution for $\delta = 1.0$ or even $\delta=0.1$. Beyond this, we observe either an over-regularization (for $\delta_2 > 1.0$), or a loss in the convergence of the iterative method used to invert the discretization matrices. This is further shown on table \ref{tab:ksp_brain}, where the GMRES space dimension is depicted as $\delta_2$ varies, and one observes how it increases when decreasing the stability weight too much, reaching a loss of convergence for $\delta_2 = 10^{-3}$. This is additional numerical evidence to support the use of the stabilizer in realistic conditions.}

\begin{figure}[!htbp]
\centering
\includegraphics[width=\textwidth]{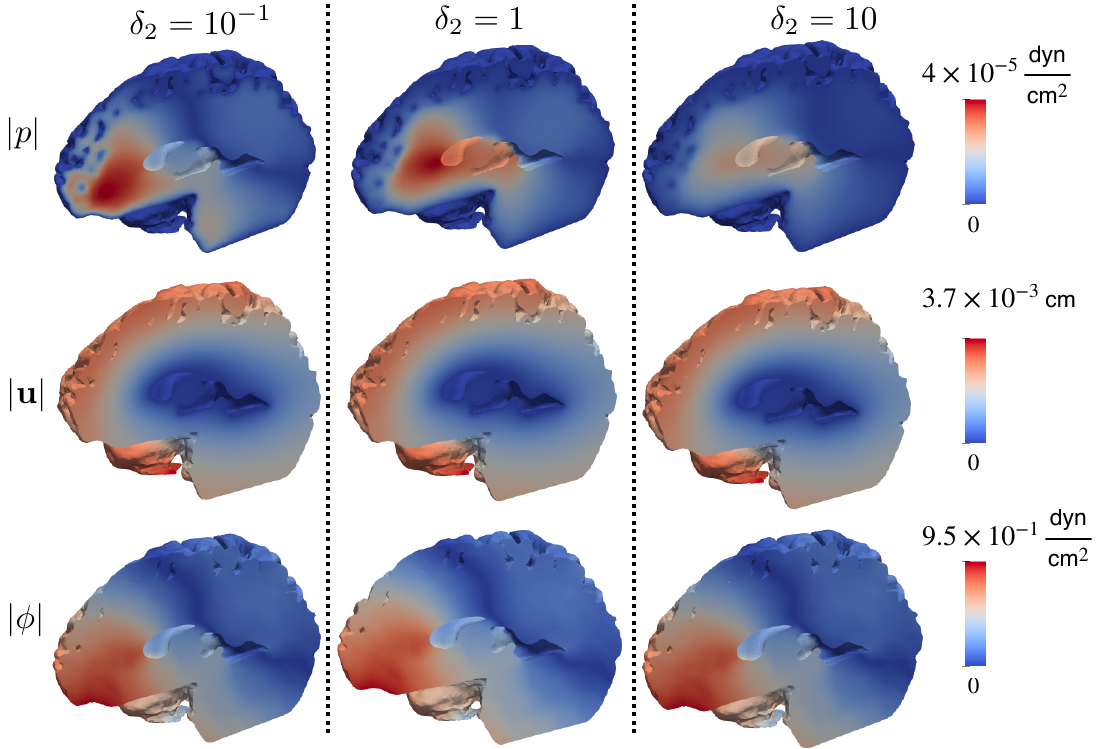}
\centering
\caption{\rrevision{Example 3 (brain geometry): Numerical results for the magnitude of displacement (top), pressure (middle), and total pressure (bottom) for three selected choices of the stabilization parameter $\delta_2$.}}
\label{fig:testBrain_results}
\end{figure}

\begin{table}[H]
\centering
\caption{\rrevision{Impact of the pressure stabilization 
parameter $\delta_2$ on the number of iterations required for the solver to convergence. The dash symbol (--) indicates that no convergence was reached.}}
\begin{tabular}{@{}cccccccc@{}}
\toprule
$\delta_2$  & 0   & $10^{-3}$ & $10^{-2}$ & $10^{-1}$ & $1.0$ & $10$  \\ \midrule
GMRES iterations & -- & --   & 260       & 230       & 197   & 180     \\ \bottomrule
\end{tabular}
\label{tab:ksp_brain}
\end{table}

%\begin{figure}[!h]
%\centering
%\includegraphics[width=\textwidth]{./fig/brain_omega.pdf}
%\centering
%\caption{\rrevision{Numerical test for example 3 with two frequencies, 10 Hz and 50 Hz. {\color{red} 7 pressure fields on vertical planes are shown alongside, taken from the brain center to the left border, using 1 cm spaced places.}
%{\color{red} CHECK WHAT IS VISIBLE IN THE PICTURE. ADD PHASES?}}}
%Pressure (dyn$/$cm$^2$) and displacement (cm) fields are presented in equidistant planes of the domain for both solutions.}}
%\label{fig:brain_omega}
%\end{figure}
\begin{figure}[!h]
\includegraphics[width=\textwidth]{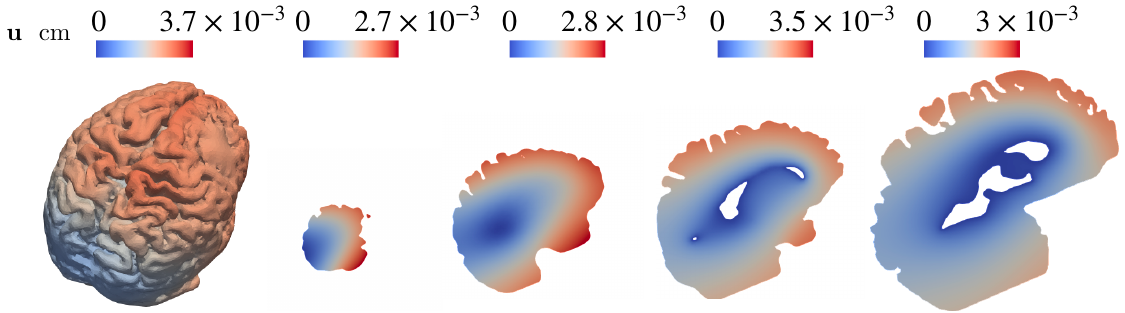} 
\caption{\rrevision{Example 3. Magnitude of the displacement solution (harmonic with frequency $\omega$) for $\omega=10Hz$. The solution is shown on $\Gamma_{\text{outer}}$ and on 4 equidistant planes, separated by 2 centimeters each.}}
\label{fig:brain_omega_10_u}
\end{figure}
\begin{figure}[!h]
\centering
\includegraphics[width=0.8\textwidth]{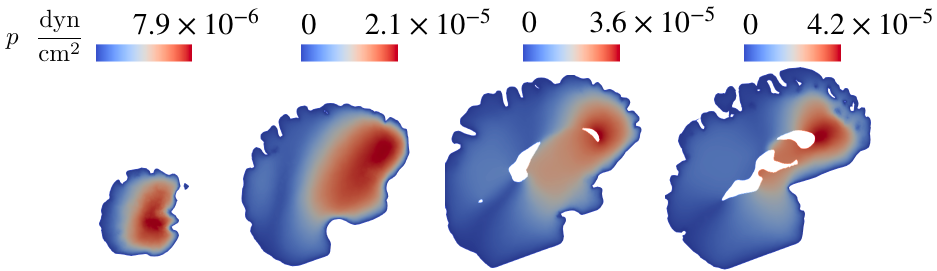}
\centering
\caption{\rrevision{Example 3. Magnitude of the pressure solution (harmonic with frequency $\omega$)for $\omega=10Hz$. The solution is shown on 4 equidistant planes, separated by 2 centimeters each.}}
\label{fig:brain_omega_10_p}
\end{figure}

\rrevision{Finally, Figures \ref{fig:brain_omega_10_u}, \ref{fig:brain_omega_10_p}, \ref{fig:brain_omega_50_u}, and \ref{fig:brain_omega_50_p} show selected views of the numerical solution (displacement and pressure fields), using $\delta_2 = 1$ and for $\omega = 10$ Hz and $\omega = 50$ Hz.}

\begin{figure}[!h]
\centering
\includegraphics[width=\textwidth]{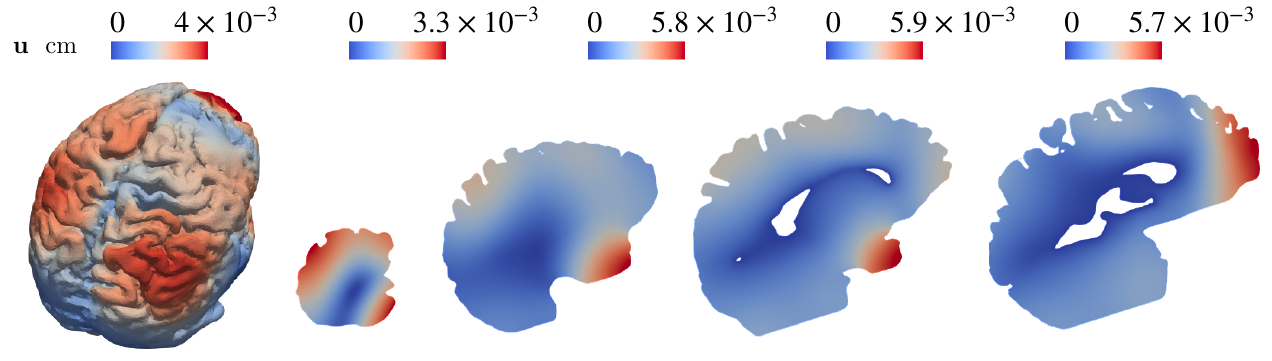}
\centering
\caption{\rrevision{Example 3. Magnitude of the displacement solution (harmonic with frequency $\omega$) for $\omega=50Hz$. The solution is shown on $\Gamma_{\text{outer}}$ and on 4 equidistant planes, separated by 2 centimeters each.}}
\label{fig:brain_omega_50_u}
\end{figure}

\begin{figure}[!h]
\centering
\includegraphics[width=0.8\textwidth]{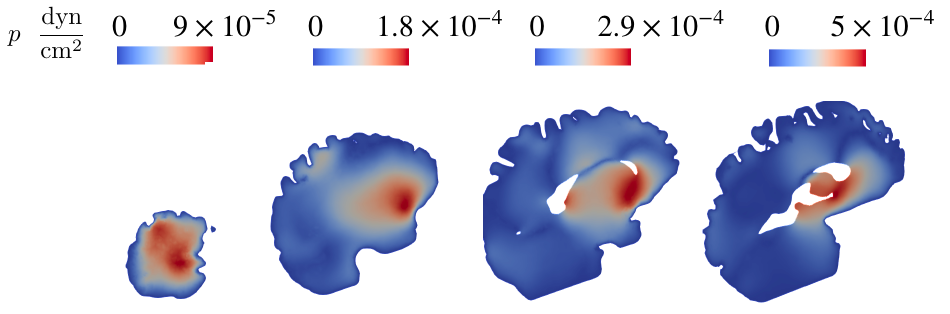}
\centering
\caption{\rrevision{Example 3. Magnitude of the pressure solution (harmonic with frequency $\omega$)for $\omega=50Hz$. The solution is shown on 4 equidistant planes, separated by 2 centimeters each.}}
\label{fig:brain_omega_50_p}
\end{figure}

\section{Conclusions}\label{sec:conclusions}
This paper proposes and analyzes a stabilized finite element method for the numerical solution of the Biot equations in the frequency domain utilizing a total-pressure formulation. We focus on the case of equal-order finite elements, introducing additional stabilization terms to cure numerical instabilities. Moreover, an additional  Brezzi-Pitkäranta stabilization is introduced to enhance robustness concerning the discontinuities of material permeability.

The first contribution of this work is the detailed numerical analysis, in the continuous and the discrete settings,  of the total pressure formulation. In particular, using the Fredholm alternative \cite{evans10,ST07}, and the T-coercivity properties of the variational form \cite{ciar12}, we show that the well-posedness results of \cite{RORR16} in the time domain case can be extended to the frequency regime. The second contribution concerns the proposed
stabilization, which allows for enhancement of the robustness of the numerical method in a wide range of tissue permeability and also in the presence of discontinuities. 

Since the additional stabilization introduces a second-order consistency error, optimal convergence can be proven only for linear equal-order finite elements.  However, in practical situations, it is important to observe that the stabilization can be introduced only where required (i.e., regions of very low permeability), thus expecting numerical results of better quality than those suggested by the theoretical expectation. An optimal choice of the stabilization parameters depending on the local solutions or considering local error estimators is the subject of ongoing research and is out of the scope of this work.

The proposed method has been validated on simple examples against analytical solutions, as well as considering a layered domain with varying permeability, and an example of a brain geometry obtained from medical imaging. Future directions of this research will consider the application of the scheme in the context of inverse problems for parameters or state estimation.

\section*{Acknowledgments}

This research is funded by the Deutsche Forschungsgemeinschaft (DFG, German Research Foundation) under Germany’s Excellence Strategy - MATH+: The Berlin Mathematics Research Center [EXC-2046/1 - project ID: 390685689]. The third author would like to acknowledge the financial support from the project DI VINCI Iniciación PUCV 039.482/2024, and the support of the student Mauricio Portilla concerning a few advances on the software used for the numerical experiments within this work. The Authors are grateful to Prof. Dr. Ingolf Sack (Department of Radiology, Charité, Berlin) for providing the
brain surface MRI data and to Christos Panagiotis Papanias and Prof. Vasileios Vavourakis for the segmentation of the medical images.

\bibliographystyle{elsarticle-num}
\bibliography{reference}
\end{document}